\documentclass[11pt]{article}
\usepackage{booktabs} 
\usepackage[ruled]{algorithm2e} 

\SetAlFnt{\small}
\SetAlCapFnt{\small}
\SetAlCapNameFnt{\small}
\SetAlCapHSkip{0pt}
\IncMargin{-\parindent}

\usepackage[utf8]{inputenc}
\usepackage[english]{babel}

\usepackage[letterpaper,margin=1.25in]{geometry}

\usepackage{amsmath,amsfonts,amsthm,amssymb,mathrsfs,dsfont} 
\usepackage{mathtools}
\usepackage{graphicx}
\usepackage[colorinlistoftodos]{todonotes}
\usepackage{verbatim}

\newcommand{\bE}{\mathbb{E}}

\newcommand{\sparagraph}[1]{{{\em #1}}\ }

\newcommand{\E}{\bE}      

\newcommand{\KL}{\mathop{\bf KL\/}}
\newcommand{\bone}{\boldsymbol{1}}
\DeclareMathOperator{\Var}{Var}
\DeclareMathOperator{\Cov}{Cov}
\DeclareMathOperator{\sech}{sech}

\DeclareMathOperator{\NAE}{\text{NAE}}

\newtheorem{theorem}{Theorem}[section]
\newtheorem*{namedtheorem}{\theoremname}
\newcommand{\theoremname}{testing}

\newtheorem{lemma}[theorem]{Lemma}

\newtheorem{fact}[theorem]{Fact}

\newtheorem{corollary}[theorem]{Corollary}

\newtheorem{conjecture}{Conjecture}
\newtheorem*{question*}{Question}

\theoremstyle{definition}

\newtheorem{defn}[theorem]{Definition}
\newtheorem{remark}[theorem]{Remark}

\theoremstyle{plain}

\begin{document}
\title{A Phase Transition in Arrow's Theorem}
\author{Frederic Koehler\thanks{Department of Mathematics, Massachusetts Institute of Technology. Supported in part by NSF CCF-1453261.} \and Elchanan Mossel\thanks{Department of Mathematics and IDSS, Massachusetts Institute of Technology. Supported by NSF Award DMS-1737944, by
ARO MURI W911NF1910217 and by a Simons Investigator Award in Mathematics (622132)}} 
\date{}


\maketitle

\begin{abstract}
Arrow's Theorem concerns a fundamental problem in social choice theory: given the individual preferences of members of a group, how can they be aggregated to form rational group preferences? Arrow showed that in an election between three or more candidates, there are situations where any voting rule satisfying a small list of natural ``fairness'' axioms must produce an apparently irrational intransitive outcome. Furthermore, quantitative versions of Arrow's Theorem in the literature show that when voters choose rankings in an i.i.d.\ fashion, the outcome is intransitive with non-negligible probability.

It is natural to ask if such a quantitative version of Arrow's Theorem holds for non-i.i.d.\ models. 
To answer this question, we study Arrow's Theorem under a natural non-i.i.d.\ model of voters inspired by canonical models in statistical physics; indeed, a version of this model was previously introduced by Raffaelli and Marsili in the physics literature. This model has a parameter, temperature, that prescribes the correlation between different voters. We show that the behavior of Arrow's Theorem in this model undergoes a striking phase transition: in the entire high temperature regime of the model, a Quantitative Arrow's Theorem holds showing that the probability of paradox for any voting rule satisfying the axioms is non-negligible; this is tight because the probability of paradox under pairwise majority goes to zero when approaching the critical temperature, and becomes exponentially small in the number of voters beyond it. We prove this occurs in another natural model of correlated voters and conjecture this phenomena is quite general.
\end{abstract}


\section{Introduction}
Arrow's Theorem concerns a fundamental problem in social choice theory: given the individual preferences of members of a group, how can these preferences be aggregated to form rational group preferences?
This problem is often discussed in the context of voting, where the goal is for society to choose between different candidates in an election based upon the ranked preferences of individual voters. The underlying problem is considerably more general and has also attracted renewed attention in computer science, see e.g. \cite{faliszewski2010ai,rothe2015economics} -- a key example outside of the classical voting context involves the aggregation of search results from different experts \cite{dwork2001rank,cohen1998learning}. 

Already in the 18th century, the Marquis de Condorcet \cite{de2014essai} considered the problem of aggregating votes and observed the following paradox: in majority-based pairwise elections between three candidates $A$, $B$, and $C$ it is possible that $A$ beats $B$, $B$ beats $C$, and $C$ beats $A$, so society's preferences may be intransitive if they are determined by pairwise majority elections. 
Arrow's Impossibility Theorem \cite{Arrow:50,Arrow:63} shows that this phenomenon is very general: there is no way to aggregate individual preferences in a way which guarantees a transitive outcome (i.e. a consistent ranking of candidates) and satisfies Independence of Irrelevant Alternatives (IIA) and Unanimity\footnote{IIA says that the aggregated relative order of $A$ and $B$ is a function only of individual preferences between $A$ and $B$ (e.g. if society ranks $A$ over $B$, an individual swapping the ranking of $C$ vs. $D$ does not affect this).
Unanimity is the condition that if all voters prefer $A$ to $B$, then the aggregated preference must also rank $A$ over $B$. Unanimity rules out constitutions where one candidate always wins/loses; in the statement of quantitative versions of Arrow's Theorem, this assumption is dropped and the possibility of a constant winner/loser is listed explicitly.} except for a ``dictator'' function which ignores the preferences of all but a single person. In other words, for any aggregation scheme except for dictator and satisfying IIA and Unanimity, there exists \emph{some} setting of individual ranked preferences such that the outcome of the election is intransitive. 

The extent to which Arrow's Theorem applies in practice has been debated extensively. See e.g. \cite{gehrlein2006condorcet,kurrild2001empirical} where some real examples of the Condorcet paradox are noted, such as cyclical voter preferences between three candidates for Prime Minister of Denmark.
A major question of concern is whether the situations where intransitivity occurs are atypical, i.e. unlikely to occur in realistic scenarios. A priori, in an election with $n$ voters and three candidates Arrow's Theorem only guarantees that one out of $6^n$ possible voting profiles -- an exponentially small proportion -- leads to intransitivity. This raises the following natural question, studied in previous work: \emph{if we ignore a $o(1)$ fraction of possible voter preferences,
does the conclusion of Arrow's Theorem still hold?}

\sparagraph{Quantitative Versions of Arrow's Theorem.}
A long line of work in quantitative social choice theory has sought to answer this question, by studying the extent to which the above stated guarantee for Arrow's Theorem can be improved. One of the earliest works along these lines was in 1952, when Guilbaud \cite{Guilbaud:66} analytically computed the asymptotic proportion of voting profiles under which pairwise majority avoids the Condorcet paradox. Guilbaud determined that for a three-candidate election, as the number of voters goes to infinity, the answer is approximately $91.2 \%$. Phrased in a more probabilistic language, if all voters independently pick a uniformly random ordering of the candidates then the probability of a Condorcet paradox occuring is over $8 \%$. A number of works considered related versions of this problem --- see e.g. \cite{black1958theory,demeyer1970probability} and several other references listed in \cite{kurrild2001empirical,gehrlein2006condorcet}.

The result of Guilbaud was restricted to pairwise majorities and so it was only a quantitative analogue of Condorcet's result as opposed to the general 
Arrow's Theorem.
Using tools from Fourier analysis on the hypercube, Kalai \cite{kalai2002fourier} (and a follow-up work of Keller \cite{Keller:09}) proved the first quantitative analogue of Arrow's theorem under an assumption that the aggregation rule is perfectly balanced. In other words, under this assumption Kalai showed that any voting rule which satisfies IIA and is at least $\delta$-far from any dictator must admit an intransitive outcome with probability at least $\epsilon(\delta)$ which is \emph{independent of $n$}. 

Unfortunately, the techniques used in \cite{kalai2002fourier,Keller:09} relied heavily on the balance assumption that the probability one candidate beats another is \emph{exactly} equal to $1/2$, leaving open the possibility that unbalanced constitutions could still avoid paradox with high probability. Finally,
Mossel \cite{mossel2012quantitative} removed the assumption of balance and proved the desired quantitative generalization of Arrow's Theorem, making use of powerful analytic tools like reverse hypercontractivity \cite{Borell:82} and the Invariance Principle \cite{mossel2005noise}. In particular, ignoring $o(1)$ fraction of possible voter preferences cannot avoid the conclusion of Arrow's Theorem. 

Repeating the same proof with more general hypercontractive estimates in the work \cite{mossel2013reverse} allowed to prove the
appropriate quantitative analogue of Arrow's Theorem in the setting where voter preferences are still i.i.d., but can be sampled from an arbitrary distribution with full support on the set of rankings. Some related lines of work in quantitative social choice theory include quantitative versions of the celebrated Gibbard-Satterthwaite Theorem \cite{Gibbard:73,Satterthwaite:75} on the manipulability of elections, see e.g. \cite{friedgut2011quantitative,isaksson2012geometry,mossel2015quantitative}, and quantitative results for judgement aggregation \cite{nehama2013approximately,filmus2019testing}.

\sparagraph{On The Independence Assumption.} The Quantitative Arrow's Theorem discussed above shows that when individuals choose their preferences uniformly at random and independent of each other, any aggregation rule which satisfies IIA and is far from a dictator is intransitive with positive probability independent of $n$. 
However, the assumption that voters choose their preferences independently of others is somewhat problematic in real voting scenarios (as mentioned in e.g. \cite{mossel2013reverse}). We know that individuals usually do not make their choices in a vacuum but instead are influenced by their interactions with other people --- both directly with members of their immediate social circle and indirectly with others through the internet, mass media, polling, etc. and it is has been suggested that such mechanisms could reduce the probability of intransitive outcome in practice (see e.g. \cite{kurrild2001empirical}).

To investigate this problem, we should modify our previous question and ask: \emph{if we allow the distribution of votes between different individuals to be correlated, does the conclusion of Arrow's Theorem still hold if we can ignore a probability $o(1)$ fraction of outcomes?} Unfortunately, in greatest possible generality this question is too broad; for example, if we allow for voters to be so correlated that they all vote in exactly the same way there is obviously no possibility of paradox and the problem of aggregation is not interesting. Also, at a technical level the arguments used in the works \cite{kalai2002fourier,Keller:09,mossel2012quantitative} rely quite strongly on techniques and results from discrete Fourier analysis over product measures, which made it unclear if they could say anything about the case where voters interact.  

\sparagraph{Statistical Mechanics Models for Consensus Formation.} What is a good model of interacting voters? In a separate line of work, also inspired by Condorcet's paradox and Guilbaud's asymptotic calculation, physicists studied group opinion formation from a statistical mechanics perspective --- see e.g. \cite{galam1997rational,raffaelli2005statistical,columbu2008nature} and \cite{gehrlein2006condorcet} for a discussion of this work in the context of the broader social choice literature. 
Notably, Raffaelli and Marsili \cite{raffaelli2005statistical} introduced a more complex model of voter interaction where voters both want to agree with the majority consensus of society and have their own random preference. 
Using heuristic methods, they described the phase diagram of their model and how to estimate the probability of Condorcet paradox. This analysis is specific to pairwise majority elections, so it doesn't give any analogue of Arrow's Theorem or tell us what happens for other models of elections (e.g. pairwise elections under an electoral college system). 

\sparagraph{Our Contribution.}
In this paper, we aim to prove best-of-both-worlds results: we establish versions of the full Quantitative Arrow's Theorem under the more complex models of voter interaction inspired by statistical mechanics. First, we study the same type of (mean-field) interaction model as \cite{raffaelli2005statistical} and give a very precise (and rigorous) analysis of the behavior of the general voting schemes in this model, in terms of the parameter $\beta$ (referred to as inverse temperature in the context of statistical mechanics). Our analysis reveals a phase transition for the Quantitative Arrow's Theorem exactly at the critical temperature (i.e. the natural phase transition point) of the model and shows as a byproduct that whenever pairwise majority elections suffer from a $\Omega(1)$ probability of paradox, this actually extends to \emph{all voting rules} which satisfy IIA, are not close to dictator, and do not fix the winner or loser of any pairwise election. So pairwise majority is threshold-optimal in this model.

We formulate a precise conjecture stating that this behavior holds in more general models and give evidence for this by verifying the conjecture in a second model which otherwise exhibits qualitatively different behavior. In this model, correlations are all local and the model does not exhibit a phase transition at any temperature; accordingly, the Quantitative Arrow's Theorem we prove in this setting shows that the probability of paradox for any constitution satisfying IIA is $\Omega(1)$ for all fixed $\beta \ge 0$, unless the constitution is close to a dictator or constant on a pairwise election.

\subsection{Our Model}
In this section we first describe a general family of models we consider for correlated voters, which is closely related to both classical models of random permutations and fundamental models of correlated spins from statistical physics, and then the two special cases we will study in detail in this work. As mentioned before, the ``mean-field'' version is the same as the interaction model in \cite{raffaelli2005statistical} and some additional justification for this model can be found there.

\sparagraph{Model on general graphs.}
Let $\mathfrak{S}_q$ denote the symmetric group (i.e. set of permutations) on $q$ elements, which corresponds to the possible individual voter preferences in an election with $q$ candidates. Let $d_{\tau}$ denote the Kendall's Tau distance between permutations, i.e. $d_{\tau}(\pi_1,\pi_2)$ is the total number of pairs $(i,j)$ such that $\pi_1(i), \pi_1(j)$ are in the opposite order from $\pi_2(i),\pi_2(j)$. 
 The previous discussion motivates the definition of the following energy function for a society of $n$ voters which interact with neighbors on a graph
\[ E(x) = 2\sum_{i \sim j} d_{\tau}(x_i,  x_j) \]
where the sum ranges over edges $(i,j)$ in the graph,
and consideration of the pairwise graphical model (Markov Random Field) over voting profiles $x \in \mathfrak{S}_q^n$ given by $e^{- \beta E(x)}$; in other words, we will consider a random vector of preferences $X \sim Q$ where
\begin{equation}\label{eqn:general-model}
Q(x) = \frac{1}{Z} \exp\left(-2 \beta \sum_{i \sim j} d_{\tau}(x_i, x_j)\right)
\end{equation}
is a probability measure on $\mathfrak{S}_q^n$, $\beta \ge 0$ is the \emph{inverse temperature} which controls the strength of interactions and $Z$ is a normalizing constant. The factor of 2 here is just to maintain consistency with \cite{raffaelli2005statistical}.
Note that in this model the marginal law of a single coordinate is uniform over $\mathfrak{S}_q$,
however the coordinates are no longer independent. 
Some justifications for this model: 
\begin{enumerate}
    \item Motivation: this reweights the uniform measure towards low-energy configurations so that neighbors are less likely to disagree, which seems like a more plausible behavior for voters who interact in a social network than if their votes are uncorrelated.
    \item Connections to the Mallows model: the Mallows model under Kendall's tau \cite{mallows1957non} is probably the most popular and well-studied model of a distribution over permutations, used in numerous fields including economics, psychometrics, and machine learning (see e.g. \cite{mallows1957non,doignon2004repeated,lu2011learning}), and for which there has been a lot of recent progress in inference and learning  (see e.g. \cite{braverman2009sorting,awasthi2014learning,liu2018efficiently}).
    If we consider the general model above in the $n = 2$ case where the graph has a single edge, the conditional law of one spin given a fixing of the other is exactly the Mallows model.
    \item Connections to Ising models: the $q = 2$ case of this model is the Ising model on a general graph, one of the most important models in statistical mechanics; this model and variants have been successfully applied in numerous other contexts including biology (e.g. \cite{bialek2012statistical}), image segmentation \cite{li2009markov}, machine learning (e.g. \cite{hinton2012practical}), and dynamics in social networks (e.g. \cite{montanari2010spread,lynn2018maximizing}). See also the discussion in \cite{raffaelli2005statistical}.
    \item Maximum entropy principle: This is the maximum entropy distribution among all probability distributions $P$ on $\mathfrak{S}_3^n$ achieving the same value of $\E_P[\sum_{i \sim j} d_{\tau}(X_i, X_j)]$;
    in statistical mechanics and Bayesian statistics, taking the maximum-entropy distribution given observed constraints is generally considered to be the most natural choice of model (see Jaynes's principle \cite{jaynes1957information}).
\end{enumerate}
We also note that this model has an interpretation as the equilibrium distribution of a simple Markov chain called Glauber dynamics or Gibbs sampling \cite{levin2017markov}; a discrete time version of the chain chooses at every step a uniformly random coordinate $i$ and resamples $X_i$ from the conditional law given all other entries of $X$; the conditional law is a tilt (see Section~\ref{sec:tau-inner-product}) of the uniform measure on $\mathfrak{S}_q$ in the direction of the (appropriately-defined) average vote of its neighbors. 

\sparagraph{Mean-field model \cite{raffaelli2005statistical}.} In statistical physics, when faced with the problem of understanding classes of models like \eqref{eqn:general-model}, one often starts by solving the \emph{mean-field model} where the underlying graph is complete, so that all pairs of vertices interact symmetrically. 
Following this principle, in this work we primarily focus on understanding the complete graph (mean-field) model.
In the $q = 2$ case, the complete graph model is known as the Curie-Weiss model and it is a fundamental example of a solvable Ising model. It is well-known that in a variety of ways the behavior of other Ising models on large degree graphs (e.g. the Ising model on the square lattice $\mathbb{Z}^d$ or on a $d$-regular random graph for $d$ large) is similar to the behavior of the Curie-Weiss model -- see e.g. \cite{ellis2007entropy,parisi1988statistical,basak2017universality}. Explicitly, the distribution of the mean-field model we consider is
\begin{equation}\label{eqn:mean-field-intro}
Q(X = x) = \frac{1}{Z} \exp\left(-\frac{\beta}{n} \sum_{i=1}^n \sum_{j = 1}^n d_{\tau}(x_i,x_j)\right)
\end{equation}
where $x \in \mathfrak{S}_q^n$ and the $1/n$ scaling is the same as in the $q = 2$ case, the Curie-Weiss model (where this is the scaling at which the model exhibits its phase transition from high to low temperature \cite{ellis-newman}). In this case, the Glauber dynamics at every step picks a coordinate $i$ uniformly at random and then resamples the spin $X_i$ in a way which is slightly correlated with the average vote of the rest of society\footnote{As an additional complication, in the dynamics we could suppose that the agents also have a fixed inherent bias in the ranking they choose. In \cite{raffaelli2005statistical} they consider the effect of choosing such a bias randomly (a ``random field'' model) and observed that this can increase the critical $\beta$. For simplicity, and following the main focus in the quantitative social choice literature (e.g. \cite{kalai2002fourier,mossel2012quantitative}), we focus on the case with no external field, though many parts should extend to the case with biases (e.g. in \cite{mossel2013reverse} it was shown how to analyze the biased product measure case, using the same general argument as \cite{mossel2012quantitative} for the unbiased case).}. Informally, we expect this to be representative of the behavior of the model defined by \eqref{eqn:general-model} for general graphs with large average degree and good connectivity properties. 

\sparagraph{Perfect matching model.} In order to understand how the behavior of \eqref{eqn:general-model} may change on sparse graphs, we consider the extreme case where every node has degree one: a perfect matching. This is the sparsest graph possible without having isolated nodes. We write this model explicitly as
\begin{equation}\label{eqn:matching-intro}
Q(X = x,Y=y) = \frac{1}{Z} \exp\left(- 2\beta \sum_{i = 1}^n d_{\tau}(x_{i},y_{i}) \right)
\end{equation}
where we have labeled the pair of voters in matching $i$ so that the first voter has ranking $X_i$ and the second voter has ranking $Y_i$. The Glauber dynamics step would be to pick an index $i$ uniformly at random, then pick uniformly at random between $X_i$ or $Y_i$ and resample this spin it in a way correlated with its neighbor, so this model exhibits only local instead of global interactions.
\subsection{Our Results}
The main result of this paper is an essentially complete analysis of the $q = 3$ mean-field model and the quantitative behavior of Arrow's Theorem in this model.
Our analysis shows that $\beta = 3/4$ is the {\em critical temperature} of this model and at this point a sharp phase transition also occurs in the behavior of general voting schemes. 
To state the high-temperature result precisely, we use the following standard notation (specialized to the $q = 3$ case) to describe the aggregation scheme:
\begin{itemize}
    \item We will refer to the candidates in the election as candidates $1$, $2$, and $3$. We let $\mathfrak{S}_3$ be the symmetric group on 3 elements, i.e. the set of permutations of $1,2,3$.
    \item $f,g,h$ are Boolean functions $\{\pm 1\}^n \to \{\pm 1\}$ which represent the aggregation schemes for each pairwise election.
    \item The (random) vector of individual voter preferences is $X$ valued in $\mathfrak{S}_3^{n}$.
    \item Based on $X$, we define vectors $X^1,X^2,X^3$ all in $\{\pm 1\}^n$ where the entries of $X^1 \in \{\pm 1\}^n$ correspond to the individual preferences between candidates $1$ and $2$, $X^2$ to the preferences between candidates $1$ and $3$, and $X^3$ between $2$ and $3$. In other words, $X^1_i = 1$ if $1$ comes before $2$ in the permutation $X_i$.
    \item The aggregated preferences of society are given by the vector $(f(X^1),g(X^2),h(X^3)) \in \{\pm 1\}^3$. Since $f,g,h$ only depend on $X^1$,$X^2$,$X^3$ respectively, this automatically encodes the Independence of Irrelevant Alternatives (IIA) assumption.
    \item 
    $\NAE_3$ is defined to be the subset of $\{\pm 1\}^3$ given by removing $\{1,1,1\}$ and $\{-1,-1,-1\}$, i.e. it represents the set of vectors not satisfying the all-equals predicate. 
    Note that intransitivity occurs exactly when $f,g,$ and $h$ are all equal.
    \item We say that $f,g$ are $\epsilon$-close with respect to probability measure $P$ if $P(f \ne g) \le \epsilon$. When omitted, the measure $P$ is always the distribution of $X$ under consideration (usually the mean-field model).
    \item We define $\mathcal{F}_3$ to be the class of functions $\mathfrak{S}_3^n \to \{\pm 1\}^3$ consisting of dictators (i.e. functions that depend only on one coordinate) and functions where two of the output coordinates are constant with opposite sign. As explained above, the coordinates represent the pairwise elections so the latter functions represent constitutions where a fixed candidate is ranked top or bottom; Arrow's Theorem says any IIA rule which completely avoids a paradox must be in $\mathcal{F}_3$.
\end{itemize}
\begin{theorem}[Mean-field Quantitative Arrow's Theorem, $q = 3$\label{thm:qa-mf-3-intro}]
Fix $\beta < 3/4, \epsilon > 0$. Suppose the vector of voter preferences $X \in \mathfrak{S}_3^{n}$ is drawn from the mean field model $Q$ defined in \eqref{eqn:mean-field-intro} with $q = 3$ candidates. There exists $\delta = \delta(\epsilon,\beta) > 0$, in particular independent of $n$, such that at least one of the following occurs, for any Boolean functions $f,g,h : \{\pm 1\}^n \to \{\pm 1\}$:
\begin{enumerate}
    \item The function $X \mapsto (f(X^1),g(X^2),h(X^3))$ which maps $\mathfrak{S}_3^n \to \{\pm 1\}^3$ is $\epsilon$-close to a function in $\mathcal{F}_3$, i.e. the constitution is close to dictator or close to having a fixed top or bottom candidate.
    \item The probability of paradox (an intransitive outcome) is lower bounded by $\delta$: \[ Q((f(X^1),g(X^2),h(X^3)) \in \NAE_3) < 1 - \delta. \]
\end{enumerate}
\end{theorem}
With this notation, the result of \cite{mossel2012quantitative} is exactly the special case of the previous Theorem with $\beta = 0$. Theorem~\ref{thm:qa-mf-3-intro} is shown by establishing mutual contiguity of the mean-field model with the i.i.d. model in the entire high-temperature regime, which in turn relies upon establishing sharp concentration estimates for averages of random vectors drawn from $Uni(\NAE_3)$. As mentioned before, we show that in the low-temperature regime $\beta > 3/4$ intransitivity is avoided by taking $f,g,h$ to be majority functions, because the probability of Condorcet's paradox is exponentially small in $n$. In fact, we give an exact formula for the asymptotic probability of Condorcet's paradox in our model as a function of $\beta$, showing that the probability of paradox goes continuously to $0$ as $\beta$ goes to $3/4$:
\begin{theorem}[Generalized Guilbaud Formula, $q = 3$]\label{thm:guilbaud-3-intro}
Let $E_n$ be the event that there is a Condorcet winner under pairwise majority elections, where the vector $X$ of individual preferences is drawn from the mean-field model $Q_n$ (see \eqref{eqn:mean-field-intro}) with $q = 3$. Then the asymptotic probability of paradox is given by
\[ \lim_{n \to \infty} Q_n(E_n) = \begin{cases}
\frac{3}{2\pi} \arccos(\frac{3}{4\beta - 9}) & \text{ if $\beta < 3/4$} \\
1 & \text{ if $\beta > 3/4$.}
\end{cases}. \]
This function is graphed in Figure~\ref{fig:guilbaud}.
\end{theorem}
The proof of the $\beta < 3/4$ case is given in Theorem~\ref{thm:generalized-guilbaud} and of the $\beta > 3/4$ case is given in Corollary~\ref{corr:no-paradox}; the first result is shown by proving Gaussianity of the limiting law of the averaged vote using tools from probability theory and Fourier analysis; the second result is also proved by deriving the limiting law of the average vote, which we reduce to solving an explicit (and involved) variational problem coming from
large deviation analysis of $X \sim Uni(\NAE_3^n)$ arising from the naive mean-field approximation for $\log Z$. In both cases, the probability of a Condorcet winner converges to $1$ in the limit $\beta \to 3/4$, which strongly suggests that this is the correct answer at the critical temperature $\beta = 3/4$; however we do not provide a rigorous analysis of the critical behavior in this paper. The proof of this result also establishes that $\beta = 3/4$ is the critical temperature for the model in the standard sense for statistical physics models, which is to say that the limit of $\frac{1}{n} \log Z$ is not analytic at this point. This critical temperature was previously determined in the physics literature \cite{raffaelli2005statistical} using a heuristic argument; they also gave a way to estimate the probability of paradox (using Monte-Carlo simulations) but do not seem to have observed an analytical expression as above. 

\begin{figure}
    \centering
    \includegraphics[trim={0 9.7cm 0 1cm},clip,scale=0.25]{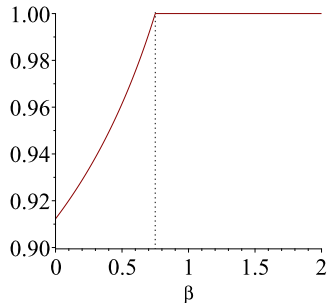}
    \caption{Asymptotic probability of a Condercet winner with pairwise majorities as a function of $\beta$ in the three candidate mean-field model. At $\beta = 0$ we reproduce Guildbaud's formula;  $\beta_c = 3/4$ (indicated by dotted line) is where the value goes to 1 and corresponds to a phase transition between high and low temperature.}
    \label{fig:guilbaud}
\end{figure}

The combination of Theorem~\ref{thm:qa-mf-3-intro} and Theorem~\ref{thm:guilbaud-3-intro} provide a striking phase transition: 
For $\beta >3/4$ we avoid paradoxes with probability $1-o(1)$ by using the most natural pairwise voting method, i.e. the majority function.
On the other hand for $\beta < 3/4$, there is no way to avoid paradoxes with probability $1-o(1)$ other than in one of the trivial ways: being $o(1)$ close to a dictator, or $o(1)$ close to a function which always outputs one of the alternatives at the top/bottom. 

\sparagraph{Mean-field model with larger $q$.} We extend the mean-field Quantitative Arrow's Theorem to the case of general $q$. In this setting, we define a constitution as a function $\mathfrak{S}_q^n \to \{\pm 1\}^{q \choose 2}$ where the coordinates of the output vector are the pairwise preferences between each pair of candidates. Here $\mathfrak{S}_q$ is the symmetric group (i.e. set of permutations) on $q$ elements.

Surprisingly, it turns out that the needed concentration/large deviations analysis (which determines the high-temperature regime for the model and the nature of the phase transition) becomes technically much more involved even when $q = 4$. The concentration problem is related to (but significantly more general than) the classical study of concentration for the number of inversions in a random permutation, i.e. Kendall's $\tau$ (see \cite{Hoeffding:63}).
Based on a new concentration estimate, we prove a result which is tight up to constants: we show the model is in a high-temperature regime for $\beta < 1/(q - 1)$, sharp for all $q$ up to a multiplicative constant of at most $3$, and prove a Quantitative Arrow's Theorem under this condition, which we now state.

Generalizing the $q = 3$ case, for any $q \ge 3$ we define a class $\mathcal{F}_q$ of functions $\mathfrak{S}_q^n \to \{\pm 1\}^{q \choose 2}$ which completely avoid paradox (i.e. satisfy IIA and transitivity and we require this for any input, not just with high probability). This class has an explicit characterization (Theorem 1.2 of \cite{mossel2012quantitative}, see also \cite{Wilson:72}) which we give now. First, note that any constitution $F$ satisfying IIA can be restricted to $F_S : \mathfrak{S}_{|S|}^n \to \{\pm 1\}^{S \choose 2}$ for any $S \subset [q]$. The constitutions $F \in \mathcal{F}_k$ are parameterized by a partition of the candidates into disjoint sets, $[q] = A_1 \cup \cdots \cup A_t$ such that:
\begin{itemize}
    \item For any $r < s$, the constitution $F$ always ranks every candidate in $A_r$ above every candidate in $A_s$.
    \item For all $A_r$ with $|A_r| \ge 3$, there exists a voter $j$ such that the restriction $F_{A_r}$ of $F$ to the candidates in $A_r$ is a dictator on voter $j$.
    \item For all $A_r$ with $|A_r| = 2$, the constitution $F_{A_r}$ given by restricting $F$ to the candidates in $A_r$ is an arbitrary non-constant function.
\end{itemize}
In particular, such constitutions always either rank one candidate above another, or follow a dictator with respect to some candidates. The Quantitative Arrow's Theorem says that any constitution with low probability of paradox is close to such a constitution:
\begin{theorem}[Mean-Field Quantitative Arrow's Theorem for $q \ge 3$]\label{thm:quantitative-arrow-allq-intro}
Fix $q \ge 3$, $\epsilon > 0$, and $\beta < \frac{1}{q - 1}$.
Let $n \ge 1$ be arbitrary and let $X$ valued in $\mathfrak{S}_q^n$ be the random vector of votes drawn from the mean-field model \eqref{eqn:mean-field-intro} with at inverse temperature $\beta$.
There exists $\delta = \delta(\epsilon,\beta,q) > 0$ such that for any constitution $F : \mathfrak{S}_q^n \to \{\pm 1\}^{q \choose 2}$ satisfying Independence of Irrelevant Alternatives (IIA), either:
\begin{enumerate}
    \item $F$ is $\epsilon$-close to a function in $\mathcal{F}_q$ with respect to the law of $X$; in particular, $F$ is close to being a dictator in some elections, or having some fixed pairwise elections.
    \item Or, the probability of paradox is lower bounded by $\delta$: if $X$ is the vector of votes drawn from the model \eqref{eqn:mean-field-intro}, the probability that the aggregated preference vector $F(X) \in \{\pm 1\}^{q \choose 2}$ satisfies transitivity is at most $1 - \delta$.
\end{enumerate}
\end{theorem}
\noindent
We also show (Section~\ref{sec:low-temp}) that the model is in a low-temperature regime when $\beta > 3/(q + 1)$; we conjecture that this bound is optimal and state a simple large deviations conjecture (Conjecture~\ref{conj:large-q-conjecture}) which would imply this and generalize the $q = 3$ case in a natural way. In \cite{raffaelli2005statistical} the authors also gave strong computational evidence that this is the correct critical temperature by solving the relevant non-convex variational problem using a gradient descent-like procedure (iterating the mean field equations).

\sparagraph{Sparse models, local interaction, and non-monotone behavior.} The above results show in the mean-field model that when the interactions between different voters is weak, the Quantitative Arrow's Theorem continues to hold true. We conjecture that this is a general phenomenon which should in fact hold in all models at sufficiently high temperature, as long as $\beta$ is normalized correctly. 
\begin{conjecture}[Universal High-Temperature Quantitative Arrow's Theorem]\label{conj:universal-arrow} 
For any $q \ge 3$ there exists $\beta' = \beta'(q) > 0$ such that the following result is true. For any $\epsilon > 0$, there exists a constant $\delta = \delta(\epsilon,q) > 0$, such that for any $n \ge 1$, $d \le n$, and any $\beta d < \beta'$, the following result holds for $X$ drawn from the Gibbs measure \eqref{eqn:general-model} at inverse temperature $\beta$ on any graph $G$ of maximum degree $d$. For any constitution $F$ satisfying IIA, either:
\begin{enumerate}
    \item $F$ is $\epsilon$-close to a function in $\mathcal{F}_q$ with respect to the law of $X$; in particular, $F$ is close to being a dictator in some elections, or having some fixed pairwise elections.
    \item Or, the probability of paradox is lower bounded by $\delta$: if $X$ is the vector of votes drawn from the model \eqref{eqn:mean-field-intro}, the probability that the aggregated preference vector $F(X) \in \{\pm 1\}^{q \choose 2}$ is transitive is at most $1 - \delta$.
\end{enumerate}
\end{conjecture}
For context, a classical result of Dobrushin (Dobrushin's uniqueness condition, \cite{dobrushin-uniqueness}) tells us that there does exist such a $\beta$ such that model is in a high-temperature phase, e.g. in the sense that Glauber dynamics mixes rapidly \cite{levin2017markov}, and the conjecture is asserting that this kind of high temperature assumption also implies the Quantitative Arrow's Theorem. Proving this conjecture appears to be a significant technical challenge. One reason is that sparse models can generate new behavior which is qualitatively different from the behavior of the mean-field model. As an illustration, we solve in the $q = 3$ case for the probability of a Condorcet paradox occurring in \eqref{eqn:general-model} on a perfect matching (i.e. where voters are paired and only interact with their paired neighbor):
\begin{theorem}[Generalized Guilbaud Formula on a Perfect Matching, Theorem~\ref{thm:matching-guilbaud}]\label{thm:matching-guilbaud-intro}
For $\beta \ge 0$ and $n \ge 1$, let $X,Y$ be drawn from the Gibbs measure $Q_n$ \eqref{eqn:matching-intro} on the matching graph with $q = 3$ candidates.
Let $E_n$ be the event that there is a Condorcet winner under pairwise majority elections. Then
\[ \lim_{n \to \infty} Q_n(E_n) = \frac{3}{2\pi} \arccos \left(\frac{-1/3 - \frac{\sinh(3\beta) + 2\sinh(\beta)}{3(\cosh(3\beta) + 2\cosh(\beta))}}{ 1 + \frac{3\sinh(3\beta) + 2\sinh(\beta)}{3(\cosh(3\beta) + 2\cosh(\beta))}}\right). \]
\end{theorem}
As we can see in Figure~\ref{fig:matching-behavior}, the probability
of Condorcet election goes up and then goes back down to its original value as we increase $\beta$. The fact that the boundary cases $\beta = 0$ and $\beta = \infty$ behave the same way is not a coincidence: when $\beta = \infty$ the two nodes in a matching are perfectly correlated, so they should act the same way as a single node in the $\beta = 0$ setting. (This limiting behavior would happen for any graph with connected components of size bounded by a constant and for any $q$.) On the other hand, the behavior in the regime $0 < \beta < \infty$ does not seem so easy to determine from first principles.

Clearly, this model exhibits the following non-monotone behavior: below some threshold ($\beta \approx 0.503$) the probability of a Condorcet paradox is a \emph{decreasing} function of $\beta$, and after the same threshold the probability of a paradox becomes an \emph{increasing} function of $\beta$. 
In particular, this illustrates that increased coordination within well-connected subcommunities can actually contribute to an increase in intransitive outcomes: e.g. similar behavior likely occurs if we consider graphs with a few dense components and with very few edges in between components.
We note that different from the mean-field model, this model does not exhibit a phase transition 
between high and low temperature phases at any value of $\beta$.
\begin{figure}
    \centering
    \includegraphics[trim={0 12.2cm 0 1cm},clip,scale=0.27]{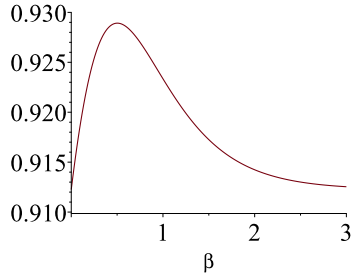}
    \caption{Asymptotic probability of a Condorcet winner with pairwise majority elections on the perfect matching, as a function of $\beta$. The maximum probability is approximately $0.929$ attained at $\beta = (1/2)\log(1 + \sqrt{3}) \approx 0.503$.}
    \label{fig:matching-behavior}
\end{figure}

Despite the notable differences between the matchings model from the mean-field model, we show that as in the mean-field model, the Quantitative Arrow's Theorem holds for the entire high-temperature regime of this model --- in this case, all $\beta \ge 0$. As above we focus on the $q = 3$ case, though an extension of the same techniques should be able to prove the result for larger $q$.
\begin{theorem}[Quantitative Arrow's Theorem on a Matching]\label{thm:qa-matching}
For $\beta \ge 0$ and $n \ge 1$, let $X,Y$ be drawn from the Gibbs measure $Q$ from \eqref{eqn:matching-intro} on the matching graph with $q = 3$ candidates. Fix $\epsilon > 0$. There exists $\delta = \delta(\epsilon,\beta) > 0$ such that at least one of the following occurs, for any Boolean functions $f,g,h : \{\pm 1\}^{2n} \to \{\pm 1\}$:
\begin{enumerate}
    \item The function $X \mapsto (f(X^1,Y^1),g(X^2,Y^2),h(X^3,Y^3))$ from $\mathfrak{S}_3^n \to \{\pm 1\}^3$ is $\epsilon$-close to a function in $\mathcal{F}_3$, i.e. the constitution is close to dictator or close to having a fixed top or bottom candidate.
    \item The probability of paradox is lower bounded by $\delta$: \[ Q((f(X^1),g(X^2),h(X^3)) \in \NAE_3) < 1 - \delta. \]
\end{enumerate}
\end{theorem}
This gives strong evidence for the validity of Conjecture~\ref{conj:universal-arrow}, as we have proved the result holds both in the densest model with very ``global'' interactions (the mean-field model) and in the matching, which is the extreme opposite case in the sense that every node has degree 1 and only local interactions are allowed to occur. At a technical level, the proof of Theorem~\ref{thm:qa-matching} is very different from the mean-field analysis (Theorem~\ref{thm:qa-mf-3-intro}), because the Gibbs measure on the matching is not mutually contiguous to the product measure for any $\beta > 0$. This makes a reduction to the i.i.d. setting impossible, because the notion of $\epsilon$-closeness is incompatible between the matching and the i.i.d. voter model. Instead, we prove the result by finding a generalization of the argument of \cite{mossel2012quantitative} for our setting, using a toolkit of reverse hypercontractive estimates, the Invariance Principle, and some new linear-algebraic arguments in Gaussian space. 
\subsection{Further Discusssion}
The standard generalization of the Ising model and Curie-Weiss model to spins valued in alphabets of size greater than two is the famous \emph{Potts model}, and it has been extensively studied in the literature --- see e.g. the review article \cite{duminil2015order}. In this model each spin takes a value in an alphabet $[q]$ and there is a fixed cost for neighboring spins to take distinct values. The mean-field Potts model has been rigorously analyzed for all $q$, see \cite{costeniuc2005complete,ellis1990limit,kesten1989behavior}. The behavior of the mean-field Potts model when $q \ge 3$ is fairly different from the Ising model ($q = 2$), see e.g. \cite{cuff2012glauber,duminil2015order} for discussion including the presence of what is known as a ``first order'' (or discontinuous) phase transition in the Potts model; in contrast, for the model considered in this paper such behavior is not expected\footnote{For example, in our analysis of our model with $q = 3$ we show the critical inverse temperature of the model is $\beta = 3/4$, which is the same threshold beyond which the all-zeros point stops being a local optimum of the related  variational problem defined in Lemma~\ref{lem:mf-eqn}; this differs from the behavior in the Potts model \cite{costeniuc2005complete,cuff2012glauber} and is related to the ``first order'' phase transition present in the latter.} \cite{raffaelli2005statistical}. Informally, the differences between the Potts model and the mean field model of \cite{raffaelli2005statistical} which we study reflect the different geometry of the simplex and the polytope corresponding to the inversion structure of permutations (see Section \ref{sec:tau-inner-product}).

The work of Starr \cite{starr2009thermodynamic} studied a different kind of mean-field model related to permutations, where the standard Mallows model over $\mathfrak{S}_q$ is considered in the limit of $q \to \infty$. As further discussed in the next section, the models we study are connected to a natural generalization of the Mallows model; it would be interesting to study all of the models discussed in this paper in the $q \to \infty$ limit, since the large $q$ behavior is not completely understood (see Conjecture~\ref{conj:large-q-conjecture} and related discussion).

There has also been a lot of interest in the computational tractability of dealing with intransitive preferences. Given a collection of intransitive pairwise preferences, finding the closest set of preferences induced by a permutation is the NP-hard \emph{feedback arc-set for tournaments} (FAST) problem; the \emph{Kemeny-Young rank aggregation} voting scheme is a special case of weighted FAST. 
See \cite{ailon2008aggregating,kenyon2007rank} for approximation algorithms for FAST, \cite{alon2009fast} for subexponential time algorithms, and \cite{braverman2009sorting} for algorithms for a closely related average case problem (MLE under the Mallows model). In our analysis we introduce a natural generalization of the Mallows model, the Inversion Tilt Model, and study its normalizing constant, which in the zero-temperature limit becomes weighted FAST. 
\section{Preliminaries}
In the Introduction, we discussed the relevant background and notation for Arrow's Theorem; here we discuss some other relevant background for the proofs. In sections 2.1 and 2.2, we recall some useful tools from probability theory which are used in the analysis of the mean-field model. In section 2.3 we note a simple but very useful way to define the Kendall's Tau distance in terms of an embedding of the permutation group into Euclidean space, which will be used throughout the paper. Finally, in the analysis of the matching model we will need some further background material (hypercontractive estimates, the Invariance Principle, Schur complement formulae) but we defer further discussion of those preliminaries to Section~\ref{sec:matching} where they are used.

\subsection{Notation} For any $\mu \in \mathbb{R}^d$ and $\Sigma : d \times d$ a positive semidefinite matrix, we let $N(\mu,\Sigma)$ denote the multivariate Gaussian distribution with mean $\mu$ and covariance $\Sigma$. For $v \in \mathbb{R}^d$, we let $\|v\| = \sqrt{\sum_i v_i^2}$ denote the Euclidean norm of $v$. For a finite set $S$, we let $Uni(S)$ denote the uniform measure on $S$,  where each element has probability mass $1/|S|$.

\subsection{Cumulant Generating Function and Gibbs Variational Principle}
The rigorous analysis of the mean-field model involves a few definitions from large deviations theory and the study of concentration inequalities which we briefly recall here: see \cite{ellis2007entropy,dembo2011large} for more. For $X \sim Q$ a mean-zero random vector defined over $\mathbb{R}^d$, its \emph{cumulant generating function} (CGF) evaluated at $\lambda \in \mathbb{R}^d$ is given by 
\[ \Psi(\lambda) = \log \E[e^{\langle \lambda, X \rangle}]  \]
provided that the expectation exists.
For $\sigma \ge 0$, we say that random vector $X$ is $\sigma^2$-\emph{sub-Gaussian} if the CGF exists everywhere and the inequality
\[ \Psi(\lambda) \le \sigma^2 \|\lambda\|^2/2 \]
is satisfied for all $\lambda \in \mathbb{R}^d$. The \emph{Chernoff bound} states that
\[  Q(\langle X, w \rangle > t) = Q(\exp \langle X, w \rangle > e^t) \le \exp \inf_{\gamma \ge 0} \left[\Psi(\gamma \langle X, w \rangle) - \gamma t\right]. \]
Note that if $X$ has sub-Gaussian constant $\sigma^2$, then this gives similar concentration estimates as when $X \sim N(0, \sigma^2 Id)$, justifying the terminology. Though we will not explicitly use this fact, in several important situations the Chernoff bound is (asymptotically) optimal because there is a matching lower bound, as in e.g. Cramer's Theorem \cite{ellis2007entropy,dembo2011large}.

The following variational principle is useful for analyzing cumulant generating functions and, more generally, for evaluating normalizing constants such as $\log Z$ from the definition of models like \eqref{eqn:general-model}. In particular, it expresses that exponentially reweighted measures with densities of the form $e^{f(X)}$ optimize a tradeoff between maximizing $\E_P[f(X)]$ and minimizing relative entropy.
\begin{lemma}[Gibbs variational principle, Lemma 4.10 of \cite{van2014probability}]\label{lem:gibbs-variational-general}
Let $X$ be a random variable on an arbitrary probability space $(\Omega,\mathcal{F},Q)$ and suppose that $f(X)$ is a measurable, real-valued function satisfying $\E_Q[e^{f(X)}] < \infty$. Then
\begin{equation}\label{eqn:gibbs-principle}
\log \E_Q[e^{f(X)}] = \sup_{P} [\E_P[f(X)] - \KL(P,Q)] 
\end{equation}
where $P$ ranges over all probability measures absolutely continuous with respect to $Q$, and $\KL(P,Q) = \E_P[\log \frac{dP}{dQ}(X)]$ is the \emph{relative entropy}, also known as \emph{Kullback-Liebler divergence}. Furthermore, the supremum is attained by the probability measure $P$ with density $\frac{dP}{dQ}(x) \propto e^{f(x)}$. 
\end{lemma}
When $X$ is valued in a cartesian product of sets like $\mathcal{X}_1 \times \cdots \times \mathcal{X}_n$ and $Q$ is a product measure, the \emph{naive mean-field approximation} (see e.g. \cite{parisi1988statistical,ellis2007entropy,eldan-gross}) is defined by restricting the right hand side of \eqref{eqn:gibbs-principle} to product measures. This approximation is always a lower bound on the true right hand side of \eqref{eqn:gibbs-principle}, is exact when $f$ is linear, and is useful for estimating the normalizing constant in classical models related to ours like the Curie-Weiss model \cite{ellis2007entropy}.
\subsection{Convergence of Measures}
We will need some fundamental tools from probability theory which we recall here; a general reference for this material
is \cite{billingsley2013convergence} or \cite{durrett2019probability}.
\begin{defn}
A sequence of real-valued random variables $\{X_n\}_{n = 1}^{\infty}$ is \emph{uniformly integrable} if
\[ \lim_{\alpha \to \infty} \sup_n \E[|X_n| \bone(|X_n| \ge \alpha)] = 0. \]
\end{defn}
\begin{defn}
A sequence of random vectors $\{X_n\}_n$  converges in distribution 
to $X$ if
$\lim_{n\to\infty} \E h(X_n) = \E h(X)$
for all bounded continuous functions $h$. This is equivalent to requiring that their CDFs converge at all points where the CDF of $X$ is continuous \cite{durrett2019probability}. Note that if $X_n \to X$ then $f(X_n) \to f(X)$ for $f$ continuous (this is sometimes called the continuous mapping theorem \cite{durrett2019probability}). 
\end{defn}
\begin{theorem}[Corollary of Portmanteau Theorem, Theorem 2.1 of \cite{billingsley2013convergence}]
Suppose that $X_n \sim Q_n$, $X \sim Q$, $X_n \to X$ in distribution, and $X$ has a continuous pdf. 
Then \[ \lim_{n \to \infty} Q_n(A) = Q(A) \]
for all Borel-measurable sets $A$. In particular, the CDFs of a real-valued random variable converge.
\end{theorem}
\begin{theorem}[Multivariate CLT, \cite{durrett2019probability}]\label{thm:clt}
Suppose $X_1,X_2,\ldots$ are an i.i.d. sequence of random vectors with mean $0$ and covariance matrix $\Sigma$. Then the sequence of partial sums $\left(\frac{1}{\sqrt{n}} \sum_{i =1 }^n X_i\right)_{n = 1}^{\infty}$ converges in distribution to a random variable with law $N(0,\Sigma)$.
\end{theorem}
\begin{theorem}[Theorem 3.5 of \cite{billingsley2013convergence}]\label{thm:ui-implies-convergence}
Suppose that $X_n \to X$ in distribution and the sequence of random variables $(X_n)_{n = 1}^{\infty}$ is uniformly integrable. Then $\lim_{n \to \infty} \E X_n = \E X$.
\end{theorem}
The proof the high temperature result will be based on establishing a contiguity result; informally, contiguity is like absolute continuity for sequences of measures. More precisely, we recall the relevant definition here (see e.g. \cite{van2000asymptotic} for a reference):
\begin{defn}
Let $(P_n)_{n = 1}^{\infty}$ and $(Q_n)_{n =1 }^{\infty}$ be two sequences of probability measures defined on the same sequence of measurable spaces. We say that $Q_n$ is \emph{contiguous} to $P_n$ if for any sequence of measurable sets $A_n$ and taking $n \to \infty$, $P_n(A_n) \to 0$ implies $Q_n(A_n) \to 0$. If $P_n$ is contiguous to $Q_n$ and $Q_n$ is contiguous to $P_n$ we say the sequence is \emph{mutually contiguous}.
\end{defn}
\subsection{Kendall's Tau and its kernel structure}\label{sec:tau-inner-product}
In the body of works on quantitative social choice (e.g. \cite{kalai2002fourier,friedgut2011quantitative,mossel2012quantitative}) it has been observed that the permutation group $\mathfrak{S}_q$ can usefully be embedded into the hypercube $\{\pm 1\}^{{q \choose 2}}$ by viewing a permutation as a list of inversions. For example, when modeling the outcome of a three-party election we can identify $\mathfrak{S}_3$ with the subset $\NAE_3 \subset \{\pm 1\}^3$ consisting of the 6 vectors where not all coordinates are equal, and then pairwise elections under majority correspond to averaging these embedded vectors and taking their coordinate-wise sign. 

This embedding also plays a key role in the solution of the models considered in this paper, because it gives a useful geometric interpretation of Kendall's Tau distance in terms of inner products, which lets us reduce questions about the behavior of mean-field models on $\mathfrak{S}_q^{n}$ to large deviation questions for random vectors in Euclidean space. Here we lay out this simple inner product structure explicitly. We note that outside of the previously mentioned context in quantitative social choice, this connection has also been used in the statistics and machine learning literature in the context of the ``kernel trick'' \cite{jiao2015kendall}.
\begin{defn}
The \emph{Kendell's Tau} distance $d_{\tau}$ between two permutations is given by
\[ d_{\tau}(\pi,\pi') := \#\{ (i,j) : i < j, \bone[\pi(i) < \pi(j)] \ne \bone[\pi'(i) < \pi'(j)] \}. \]
\end{defn}
\begin{defn}
Define $\varphi : \mathfrak{S}_q \to \{\pm 1\}^{q \choose 2}$ by
\[ \varphi(\pi)_{i,j} = (-1)^{\bone[\pi(i) > \pi(j)]} \]
where the indices range over $\{i,j\} \in {q \choose 2}$ with $i < j$.
\end{defn}
\begin{lemma}\label{lem:tau-inner-product}
For any permutations $\pi,\pi' \in \mathfrak{S}_q$,
\[ \langle \varphi(\pi), \varphi(\pi') \rangle = {q \choose 2} - 2 d_{\tau}(\pi,\pi'). \]
\end{lemma}
\begin{proof}
By writing out the left hand side, we see 
\begin{align*}  \langle \varphi(\pi), \varphi(\pi') \rangle 
&= \sum_{i < j} (-1)^{\bone[\pi(i) > \pi(j)] + \bone[\pi'(i) > \pi'(j)]} \\
&= \sum_{i < j} (1 - 2\cdot \bone[\bone[\pi(i) > \pi(j)] \ne \bone[\pi'(i) > \pi'(j)]])
\end{align*}
and the last expression equals the right hand side.
\end{proof}
\begin{remark}
The particular choice of embedding $\varphi$ is not crucial;
any embedding with the inner products prescribed by Lemma~\ref{lem:tau-inner-product} works equally well for our purposes.
\end{remark}

In the analysis of the models considered in this paper, the following exponential family of distributions over permutations, which we will refer to as the Inversion Tilt Model, appear naturally. As we explain below, the normalizing constant in this model is the CGF of the random vector $\varphi(\pi)$ for $\pi \sim Uni(\mathfrak{S}_q)$, so studying the inversion tilt model is closely related to the concentration of that random vector, which plays a central role in the mean-field analysis. 
These distributions are maximum entropy distributions over permutations given fixed value of $\E \langle \lambda, \varphi(\pi) \rangle$ (this is a consequence of the Gibbs Variational Principle, see Lemma~\ref{lem:gibbs-variational-general} and Lemma~\ref{lem:gibbs-variational}).
\begin{defn}
The \emph{Inversion Tilt Model} 
with parameter $\lambda \in \mathbb{R}^{q \choose 2}$ is the probability measure $P_{\lambda}$ on $\mathfrak{S}_q$ defined by
\[ P_{\lambda}(\pi) = \frac{1}{Z} \exp\left(\langle \lambda, \varphi(\pi) \rangle\right). \]
\end{defn}
\noindent
Up to additive constant, $\log Z$ is the cumulant generating function $\Psi(\lambda)$ of the random vector $\varphi(\pi)$ with $\pi \sim Uni \left(\mathfrak{S}_q\right)$.
Here the energy function $E(\pi) = -\langle \lambda, \varphi(\pi) \rangle$ can be interpreted as giving a (possibly negative) cost $-\lambda_{ij}$ for inverting the pair $(i,j)$. Besides appearing in the analysis of the mean-field model, it also appears in the conditional law of $X_i$ given the rest of $X$, as used in the natural Markov chain for sampling, Glauber dynamics (previously discussed in the introduction).

To the best of our knowledge, this is the first time that this general model on permutations has been considered in the literature. The Inversion Tilt Model contains the celebrated Mallows model \cite{mallows1957non} and Generalized Mallows Model \cite{fligner1986distance} under the Kendall's Tau distance as special cases. Although these models are studied under other distances on permutations as well, the Mallows model under Kendall's tau is by far the most popular and well-studied.
\begin{remark}
A main technical focus of this paper is understanding the behavior of the normalizing constant $Z$ from the Inversion Tilt Model, especially upper bounds. It is not too hard to see that $Z$ is NP hard to compute as a function of $q$ and $\lambda \in \mathbb{R}^{q \choose 2}$: if we take $\|\lambda\| \to \infty$ this becomes equivalent to the NP-hard weighted Feedback Arcset in Tournaments (FAST) problem \cite{ailon2008aggregating}.
\end{remark}
\section{Mean-Field Model of a Three-Way Election}\label{sec:mf-3}
In this section we thoroughly analyze the behavior of the mean-field model for an election between three candidates: as we will see this model has a critical inverse temperature $\beta = 3/4$ and we give a detailed description of the behavior of the model both in the high temperature regime $\beta < 3/4$ and low temperature regime $\beta > 3/4$. The two candidate model is classical (Curie-Weiss model) and the reader can refer to \cite{ellis2007entropy} for a complete analysis of it; our arguments follows a similar strategy to reduce various questions about this model to large deviations problems which we must then solve.

\textbf{Notation: }
In this section, instead of invoking the previously described embedding $\varphi$ of $\mathfrak{S}_3$ into $\{\pm 1\}^3$ throughout, it will be more convenient and consistent with previous work to use the embedding
\[ \tilde{\varphi}(\pi) := \left((-1)^{\bone[\pi(1) < \pi(2)]}, (-1)^{\bone[\pi(2) < \pi(3)]}, (-1)^{\bone[\pi(3) < \pi(1)]}\right) \]
which differs just in reordering coordinates and flipping the sign of the last coordinate. In particular, this embedding has the same inner product structure as the $\varphi$ embedding as used in Lemma~\ref{lem:tau-inner-product}.
The reason for picking this embedding (as done in \cite{o2014analysis})) is that its image has a convenient description as $\NAE_3$, the subset of $\{\pm 1\}^n$ given by removing $\{(+1,+1,+1),(-1,-1,-1)\}$. Equivalently $\NAE_3$ is the set of 6 vectors satisfying the not-all-equals predicate. 
Using the identification of $\mathfrak{S}_3$ and $\NAE_3$ we can think of the model as a distribution over $x \in \NAE_3^n \subset \left(\{\pm 1\}^3\right)^n$ given by
\begin{equation}\label{eqn:mf-nae}
Q(X = x) = \frac{1}{Z} \exp\left(\frac{\beta}{2n} \sum_{i = 1}^n \sum_{j = 1}^n \langle x_i, x_j\rangle\right) 
\end{equation}
and we will usually refer to $\NAE_3$ instead of $\mathfrak{S}_3$. Likewise, in what follows $Z = Z_n(\beta)$ will always refer to the normalizing constant $Z$ in the above expression, explicitly
\begin{equation}\label{eqn:Z-mf}
Z = Z_n(\beta) = \sum_{x \in \NAE_3^n} \exp\left(\frac{\beta}{2n} \sum_{i = 1}^n \sum_{j = 1}^n \langle x_i, x_j\rangle\right).
\end{equation}

\subsection{Subcritical regime}
\subsubsection{Quantitative Arrow's Theorem}
We first recall the following quantitative version of Arrow's theorem in the i.i.d. setting. Informally it states that in an election under ``impartial culture'', i.e. where voters choices are drawn i.i.d. from uniform on $\NAE_3$, and all candidates have a positive chance of winning, then there is a positive probability of a paradox (independent of $n$) unless we have a near-dictatorship. 
\begin{theorem}[Quantitative Arrow Theorem \cite{mossel2012quantitative}, $q = 3$ case]\label{thm:quantitative-arrow-original-3}
Fix $\epsilon > 0$.
Suppose each voter votes independently and uniformly at random from $\NAE_3$ and there are $n$ voters; i.e. votes are sampled from $P = Uni(\NAE_3)^{\otimes n}$.  There exists $\delta = \delta(\epsilon) > 0$ such that at least one of the following occurs, for any Boolean functions $f,g,h : \{\pm 1\}^n \to \{\pm 1\}$:
\begin{enumerate}
    \item The function $X \mapsto (f(X^1),g(X^2),h(X^3))$ from $\mathfrak{S}_3^n \to \{\pm 1\}^3$ is $\epsilon$-close to a function in $\mathcal{F}_3$, i.e. the constitution is close to dictator or close to having a fixed top or bottom candidate.
    \item The probability of paradox is lower bounded by $\delta$: 
    \[ Q((f(X^1),g(X^2),h(X^3)) \in \NAE_3) < 1 - \delta. \]
\end{enumerate}
\end{theorem}
In this section we generalize this theorem up to the sharp threshold of $\beta_c = 3/4$ in the mean-field model of interacting voters; later we will show that past this point (low temperature) regime the analogous theorem is false because paradox is indeed avoidable. Note that the above theorem is the special case of the theorem we will prove with $\beta = 0$. The main technical step in this analysis is the following Lemma which will be proved at the end of the section:
\begin{lemma}\label{lem:w_n-ui}
Fix $\beta < 3/4$. Suppose that $Y_n := \frac{1}{\sqrt{n}} \sum_{i = 1}^n X_i$ where
$X_1,X_2,\ldots$ are drawn i.i.d. from the uniform measure on $\NAE_3$, and let $P$ be the joint law of the $X_i$. 
The sequence of random variables $W_n := \exp(\frac{\beta}{2} \langle Y_n, Y_n \rangle)$ is uniformly integrable.
\end{lemma}
Given the above Lemma, one can show $Y_n$ converges to a Gaussian with variance depending on $\beta$ (Lemma~\ref{lem:clt-magnetization}, deferred to Section~\ref{sec:guilbaud}) and from these facts establish mutual contiguity between the mean-field model and the i.i.d. model of voters:
\begin{lemma}\label{lem:mutual-contiguity}
Fix $\beta < 3/4$. There exists $f : (0,\infty) \to (0,\infty)$ a decreasing function such that the following is true for all $n \ge 1$ and $\delta > 0$. Let $P$ denote the uniform measure on $\NAE_3^{n}$ and $Q$ the mean field model defined in \eqref{eqn:mf-nae}.
If $Q(A) > \delta$ then $P(A) > f(\delta)$, and if $P(A) > \delta$
then $Q(A) > f(\delta)$.
\end{lemma}
\begin{proof}
First, we show that a lower bound of $P(A) > \delta$ implies a lower bound on $Q(A)$.
Observe that
$Q(A) = \E_Q[\bone_A] = \E_P\left[\frac{Q(X)}{P(X)} \bone_A\right]$.
By Lemma~\ref{lem:contig1} below, there exists $K = K(\delta,\beta) > 0$ such that that if $B$ is the event $\frac{Q(X)}{P(X)} > e^{-K}$ then $P(B) \ge 1 - \delta/2$. Therefore
\begin{align*} 
Q(A) = \E_P\left[\frac{Q(X)}{P(X)} \bone_A\right] \ge \E_P\left[\frac{Q(X)}{P(X)} \bone_{A \cap B}\right] 
&\ge e^{-K} P(A \cap B) \\
&= e^{-K}(P(A) - P(A \cap B^C)) 
\ge e^{-K}\delta/2.
\end{align*}
In the reverse direction, if $Q(A) > \delta$, the argument is symmetrical except that we replace the use of Lemma~\ref{lem:contig1} by Lemma~\ref{lem:contig2}.
\end{proof}
\begin{lemma}\label{lem:contig1}
Let $\delta \in (0,1)$ be arbitrary.
In the same setting as Lemma~\ref{lem:mutual-contiguity}, there exists $K = K(\delta,\beta)$ such that 
$P\left(\log \frac{P(X)}{Q(X)} > K\right) \le \delta/2$.
\end{lemma}
\begin{proof}
Since $P(x) = 1/6^n$ for all $x \in \NAE_3^n$, we have from \eqref{eqn:mf-nae} that
\begin{equation}\label{eqn:logpq}
\log \frac{P(X)}{Q(X)} = (\log Z - n \log(6)) - \frac{\beta}{2n} \left\langle \sum_i X_i, \sum_i X_i \right\rangle 
\end{equation}
where $Z$ is the normalizing constant, as in \eqref{eqn:Z-mf}. Next, observe that $\E_{P}[X_i] = (0,0,0)$ for all $i$, so by Hoeffding's inequality \cite{vershynin2018high} applied under the i.i.d. measure $P$, we can bound each of the three coordinates of $\sum_i X_i$ in absolute value by $O(\sqrt{n \log(12/\delta)})$ with probability at least $1 - \delta/6$ individually, so by the union bound the same bound holds for all three coordinates at once with probability at least $1 - \delta/2$. It remains to control $\log Z$, which follows from uniform integrability (Lemma~\ref{lem:w_n-ui}), as
\begin{equation}\label{eqn:logz-nonneg}
\log Z - n \log 6 = \log \frac{1}{6^n} Z = \log \E_{X \sim P}[\exp((\beta/2) \langle Y_n, Y_n \rangle)] \in [0, C_{\beta}]   
\end{equation}
which proves the result.
\end{proof}
\begin{lemma}\label{lem:contig2}
Fix $\delta > 0$. In the same setting as Lemma~\ref{lem:mutual-contiguity}, there exists $K = K(\delta,\beta)$ such that with probability at least $1 - \delta$,
$Q\left(\log \frac{Q(X)}{P(X)} > K\right) \le \delta/2$.
\end{lemma}
\begin{proof}
As in \eqref{eqn:logpq} in the proof of the previous Lemma, we have
\[ \log \frac{Q(X)}{P(X)} = (n \log 6 - \log Z) + \frac{\beta}{2} \left\langle \frac{1}{\sqrt{n}} \sum_i X_i, \frac{1}{\sqrt{n}} \sum_i X_i \right\rangle \]
and the first term is nonpositive by \eqref{eqn:logz-nonneg}. For the second term, we use Lemma~\ref{lem:clt-magnetization} which says that the law of $\frac{1}{\sqrt{n}} \sum_i X_i$ under $Q_n$ converges (weakly) to a Gaussian $N(0,\Sigma_{\beta})$ as $n \to \infty$.
For a random vector $Z \sim N(0,\Sigma_{\beta})$, standard Gaussian tail bounds and the union bound
imply that with probability at least $1 - \delta/4$ that the magnitude of the coordinates of $Z$ are all at most $O_{\beta}(\sqrt{\log(2/\delta)})$. Therefore, by convergence in distribution the same bound holds with probability at least $1 - \delta/2$ for the random vector $\frac{1}{\sqrt{n}} \sum_{i = 1}^n X_i$ under $Q$ as long as $n$ is sufficiently large. 
For the remaining small values of $n$, we can simply bound the term by a constant (which is allowed to depend on $\beta$).
\end{proof}
\begin{theorem}[Mean-Field Quantitative Arrow Theorem, Restatement of Theorem~\ref{thm:qa-mf-3-intro}]\label{thm:qa-mf-3}
Fix $\beta < 3/4$. Suppose the vector of voter preferences $X \in \mathfrak{S}_3^{n}$ is drawn from the mean field model $Q$ as defined in \eqref{eqn:mean-field-intro} with $q = 3$ candidates. Fix $\epsilon > 0$. There exists $\delta = \delta(\epsilon) > 0$ such that at least one of the following occurs, for any Boolean functions $f,g,h : \{\pm 1\}^n \to \{\pm 1\}$:
\begin{enumerate}
    \item The function $X \mapsto (f(X^1),g(X^2),h(X^3))$ from $\mathfrak{S}_3^n \to \{\pm 1\}^3$ is $\epsilon$-close to a function in $\mathcal{F}_3$, i.e. the constitution is close to dictator or close to having a fixed top or bottom candidate.
    \item The probability of paradox is lower bounded by $\delta$: 
    \[ Q((f(X^1),g(X^2),h(X^3)) \in \NAE_3) < 1 - \delta. \]
\end{enumerate}
\end{theorem}
\begin{proof}
Fix $\epsilon > 0$. By Lemma~\ref{lem:mutual-contiguity} (applied to the complementary event) there exists $\epsilon' > 0$  such that for any event $A$,
if $P(A) > 1 - \epsilon'$ then $Q(A) > 1 - \epsilon$.
Let $\delta' = \delta'(\epsilon') > 0$ be as specified in Theorem~\ref{thm:quantitative-arrow-original-3}. Finally,
choose $\delta = \delta(\delta')$ by Lemma~\ref{lem:mutual-contiguity} so that
if $P(A) > 1 - \delta'$ then $Q(A) > 1 - \delta$. We claim that this $\delta$ (which depends only on $\epsilon$) satisfies the claim in the Theorem. Explicitly,
suppose that the constitution given by $(f,g,h)$ is not $\epsilon$-close to any element of $\mathcal{F}_3$ under $Q$, i.e.
\[ \max_{r \in \mathcal{F}_3} Q((f(X^1),g(X^2),h(X^3)) = r(X)) \le 1 - \epsilon, \]
then the constitution is not $\epsilon'$-close to any element of $\mathcal{F}_3$ under the i.i.d. measure $P$. Therefore, by Theorem~\ref{thm:quantitative-arrow-original-3} the probability of paradox is lower bounded by $\delta'$ under the i.i.d. measure $P$, which implies a lower bound of $\delta$ under the mean-field measure $Q$.
\end{proof}
Finally, we prove the key uniform integrability result. Recall that a $\delta$-net of the unit sphere $S^{d - 1} \subset \mathbb{R}^d$ (with respect to the Euclidean metric) is a set of points $N_{\delta} \subset S^{d - 1}$ such that for every point $x \in S^{d - 1}$, there exists $y \in N_{\delta}$ with $\|x - y\|_2 \le \delta$, and that there exist $\delta$-nets of $S^{d - 1}$ of size $(3/\delta)^d$, see e.g. Corollary 4.2.13 of \cite{vershynin2018high}. 
We will use the following standard result about nets on the sphere:
\begin{lemma}[Exercise 4.4.2 of \cite{vershynin2018high}]\label{lem:net}
For $\delta \in (0,1)$, if $N_{\delta}$ is a $\delta$-net of the unit sphere in $\mathbb{R}^d$, then for any $x \in \mathbb{R}^d$, 
$\max_{u : \|u\| = 1} \langle x, u \rangle = \|x\| \le \frac{1}{1 - \delta} \max_{u \in N_{\delta}} \langle x, u \rangle$.
\end{lemma}
The following elementary inequality plays a key role in our argument.
\begin{lemma}\label{lem:cosh-inequality}
For any $a,b,c \in \mathbb{R}$,
$\frac{1}{3} (\cosh(a) + \cosh(b) + \cosh(c)) \le \exp(a^2/6 + b^2/6 + c^2/6)$.
\end{lemma}
\begin{proof}
By expanding both sides, it is equivalent to show that
\[ \frac{1}{3} \sum_{k} \frac{a^{2k} + b^{2k} + c^{2k}}{(2k)!} \le \sum_{k = 0}^{\infty} \frac{(a^2 + b^2 + c^2)^k}{6^k k!}. \]
By nonnegativity of all terms and symmetry it suffices to check this for the coefficients only involving $a$, where it reduces to $6^k k! \le 3 (2k)!$ for $k \ge 1$. This follows from induction, as $6(k + 1) \le (2k + 2)(2k + 1)$ for $k \ge 1$.
\end{proof}
\begin{proof}[Proof of Lemma~\ref{lem:w_n-ui}]
Observe that
\begin{align*} 
\E_P[W_n \bone(W_n > \alpha)] 
&= \int_0^{\infty} P(W_n 1(W_n > \alpha) > w) dw \\
&=  \alpha P(W_n > \alpha) + \int_{\alpha}^{\infty} P(W_n > w) dw. 
\end{align*}
so it suffices to show that $P(W_n > w) = O(w^{-(1 + \epsilon)})$ for any $\epsilon > 0$. Define $x > 0$ by $w = \exp((\beta/2) x^2)$ and fix $\delta > 0$ (to be taken small). Then by Lemma~\ref{lem:net} and the union bound
\begin{align*} P(W_n > w) = P(\langle Y_n, Y_n \rangle > x^2) 
&= P(\max_{u : \|u\|_2 = 1} \langle Y_n, u \rangle > x) \\
&\le |N(\delta)|\max_{u \in N_{\delta}}P( \langle Y_n, u\rangle > (1 - \delta) x).  
\end{align*}
By applying the Chernoff bound (i.e. the inequality $P(X > a) \le \E[e^{\lambda X}]/e^{\lambda a}$) 
we know that for any vector $u$,
$P(\langle Y_n, u \rangle > (1 - \delta) x) \le \min_{\lambda \ge 0} \frac{\E[\exp(\lambda \langle Y_n, u \rangle)]}{e^{\lambda(1 - \delta) x}}$.
Now we observe the following sub-gaussian bound for all unit vectors $u$: 
\begin{align*} 
\log &\E[\exp(\lambda \langle Y_n, u \rangle)] \\
&= n \log \E_{X \sim Uni(\NAE_3)}[\exp(\frac{\lambda}{\sqrt{n}} \langle X, u \rangle)] \\
&= n \log \frac{1}{3}\Big[\cosh\left(\frac{\lambda}{\sqrt{n}}(u_1 + u_2 - u_3)\right) + \cosh\left(\frac{\lambda}{\sqrt{n}} (u_1 - u_2 + u_3)\right) \\
&\qquad\qquad\qquad + \cosh\left(\frac{\lambda}{\sqrt{n}} (-u_1 + u_2 + u_3)\right)\Big] \\
&\le \frac{1}{6} \lambda^2 [(u_1 + u_2 - u_3)^2 + (u_1 - u_2 + u_3)^2 + (-u_1 + u_2 + u_3)^2]
= \frac{2}{3} \lambda^2
\end{align*}
where the inequality follows from Lemma~\ref{lem:cosh-inequality} and the last equality follows by computing the maximum of the quadratic form: explicitly,
\[ [(u_1 + u_2 - u_3)^2 + (u_1 - u_2 + u_3)^2 + (-u_1 + u_2 + u_3)^2] = u^T \begin{bmatrix} 3 & -1 & -1 \\ -1 & 3 & -1 \\ -1 & -1 & 3 \end{bmatrix} u \]
and the maximum eigenvalue of the matrix on the right hand side is $4$.
It follows that
\[ P(\langle Y_n, u \rangle > (1 - \delta)x) \le \min_{\lambda \ge 0} e^{(2/3) \lambda^2 - \lambda(1 - \delta)x} = e^{(3/8) (1 - \delta)^2 x^2 - (3/4)(1 - \delta)^2 x^2} = e^{-(3/8)(1 - \delta)^2 x^2}\]
(since the minimizer is at $\lambda = (3/4)(1 - \delta)x$)
so combining everything
\[ P(W_n > w) \le |N(\delta)| \exp(-(3/8)(1 - \delta)^2x^2) = |N(\delta)| w^{-\frac{3}{4\beta}(1 - \delta)^2}. \]
Note that $|N(\delta)|$ does not depend on $n$ since the net is over a sphere in fixed dimension (three dimensions). Finally, as long as $\beta < 3/4$ we can choose $\delta$ sufficiently small such that $\frac{3}{4\beta}(1 - \delta)^2 = 1 + \epsilon$ which proves the result.
\end{proof}
\subsubsection{Generalized Guilbaud's Formula}\label{sec:guilbaud}
In the following two lemmas we compute the limiting distribution of $Y_n$ under the mean-field Gibbs measure \eqref{eqn:mf-nae}. In \cite{raffaelli2005statistical} the authors also computed what the limiting covariance matrix should be, though they did not give a mathematically rigorous proof of this.
\begin{lemma}\label{lem:logz}
For every $\beta < 3/4$,
$\lim_{n \to \infty} (\log Z_n(\beta) - n \log(6)) = \log \sqrt{\frac{\det \Sigma_{\beta}}{\det \Sigma_{0}}}$
where
\[ \Sigma_{\beta} := \frac{1}{4\beta^2 - 15\beta + 9} \begin{bmatrix} 9 - 4\beta & -3 & -3 \\ -3 & 9 - 4\beta & -3 \\ -3 & -3 & 9 - 4\beta \end{bmatrix}. \]
and $Z_n = Z_n(\beta)$ is defined in \eqref{eqn:Z-mf}.
\end{lemma}
\begin{proof}
Let $P$ and $Y_n$ be as in the statement of Lemma~\ref{lem:w_n-ui}, so that under $P$ the random vector $Y_n$ is distributed as a normalized sum of i.i.d. samples from $\NAE_3 \subset \{\pm 1\}^3$.  
Observe that
\[ \log Z_n - n \log 6 = \log \frac{Z_n}{6^n} = \log \E_{P}\left[\exp\left(\frac{\beta}{2} \langle Y_n, Y_n \rangle\right)\right]. \]
By Lemma~\ref{lem:w_n-ui}, $\exp(\frac{\beta}{2} \langle Y_n, Y_n \rangle)$ is uniformly integrable. Furthermore, we observe from the definition of $Y_n$ that it is a normalized sum of i.i.d. random vectors $X_1,X_2,\ldots$ with mean zero and covariance matrix 
\[ \Sigma_0 = \begin{bmatrix} 1 & -1/3 & -1/3 \\ -1/3 & 1 & -1/3 \\ -1/3 & -1/3 & 1 \end{bmatrix} \]
(this is a special case of Lemma 2.1 from \cite{mossel2012quantitative}).
Therefore from the Central Limit Theorem and Theorem~\ref{thm:ui-implies-convergence} we have
\[  \lim_{n \to \infty} \log \E_{P}\left[\exp\left(\frac{\beta}{2} \langle Y_n, Y_n \rangle\right)\right] = \log \E_{G \sim N(0,\Sigma_0)}\exp\left(\frac{\beta}{2} \langle G, G \rangle\right) = \log \sqrt{\frac{\det \Sigma_{\beta}}{\det \Sigma_{0}}}. \]
The last equality follows by computing normalizing constants for Gaussian distributions, since
\[ \E_{G \sim N(0,\Sigma_0)}\left[\exp\left(\frac{\beta}{2} \langle G,G \rangle\right)\right] = \frac{1}{\sqrt{2 \pi \det \Sigma_0}} \int e^{-g^T \Sigma_0^{-1} g/2 + (\beta/2)g^T g} dg \]
and the integral is just the normalizing constant for another multivariate Gaussian distribution.
To perform the last calculation explicitly, observe that $\Sigma_0$ is the covariance matrix of $Y_n$ in the $\beta = 0$ case (as explained in detail in the proof of Lemma~\ref{lem:cov-eigenvalues}), so computing its inverse $\Sigma_0^{-1}$, we see that
under the exponential reweighting we get new inverse covariance (i.e. precision) matrix
\[ \Theta_{\beta} = \Sigma_0^{-1} - \beta I = \begin{bmatrix} 3/2 - \beta & 3/4 & 3/4 \\ 3/4 & 3/2 - \beta & 3/4 \\ 3/4 & 3/4 & 3/2 - \beta \end{bmatrix}  \]
and inverting $\Theta_{\beta}$ (using the assumption $\beta < 3/4$) gives the expression above for $\Sigma_{\beta}$,
so using the standard formula for the normalizing constant of a multivariate Gaussian distribution gives the result.
\end{proof}
A similar argument 
proves the following Lemma as well.
\begin{lemma}\label{lem:clt-magnetization}
Fix $\beta < 3/4$ and suppose $X^{(n)} \sim Q_n$ where $Q_n$ is the mean-field Gibbs measure \eqref{eqn:mf-nae} on $n$ voters. Then the sequence of random variables $Y_n := \frac{1}{\sqrt{n}} \sum_{i = 1}^n X^{(n)}_i$ converges in distribution to
$Y \sim N(0, \Sigma_{\beta})$.
\end{lemma}
\begin{proof}
Let $h$ be an arbitrary continuous, bounded function. We need to show $\lim_{n \to \infty} \E h\left(Y_n\right) = \E_{Y \sim N(0,\Sigma_{\beta})} h(Y)$.
Let $Q = Q_n$ denote the mean-field Gibbs measure on $n$ voters and $P = P_n$ denote the measure where $X^{(n)} \sim Uni(\NAE_3^{n})$ and $Y_n$ is defined the same way. Observe that
$\frac{dQ}{dP} = \frac{2^n}{Z} \exp(\frac{\beta}{2} \langle Y_n, Y_n \rangle)$
and $1 = \E_P \frac{dQ}{dP} = \frac{2^n}{Z} \E_P \exp(\frac{\beta}{2} \langle Y_n, Y_n \rangle)$ so 
\[ \frac{2^n}{Z} = \frac{1}{\E_P \exp(\frac{\beta}{2} \langle Y_n, Y_n \rangle)}. \]
Therefore,
\[ \E_Q h\left(Y_n\right) = \E_P \frac{dQ}{dP} h(Y_n) = \E_P\left[h(Y_n)\frac{\exp((\beta/2) \langle Y_n, Y_n\rangle)}{\E \exp((\beta/2) \langle Y_n, Y_n\rangle)}\right].  \]
As in the proof of Lemma~\ref{lem:logz}, by the Central Limit Theorem (Theorem~\ref{thm:clt})  we know that 
the law of $Y_n$ under $P_n$ converges (weakly) to $N(0,\Sigma_0)$ as $n \to \infty$ and we want to check uniform integrability to replace the expectation by the Gaussian one. Since $h$ is bounded it clearly suffices to check this for $W_n := \exp((\beta/2) \langle Y_n, Y_n\rangle)$, and this was proved in Lemma~\ref{lem:w_n-ui}. Then the result follows from the same calculation as the proof of Lemma~\ref{lem:logz}.
\end{proof}
The spherical symmetry of the standard Gaussian implies the following well-known formula which can be found in a variety of references, such as \cite{o2014analysis,vershynin2018high}:
\begin{lemma}[\cite{o2014analysis,vershynin2018high}]\label{lem:gaussian-sgn}
Suppose that $X \sim N(0,\Sigma)$ where $\Sigma : 2 \times 2$. Then
\[ \E[sgn(X_1)sgn(X_2)] = 1 - \frac{2 \arccos \rho}{\pi} \]
where $\rho = \frac{\Sigma_{12}}{\sqrt{\Sigma_{11} \Sigma_{22}}}$.
\end{lemma}
\begin{theorem}\label{thm:generalized-guilbaud}
For $\beta < 3/4$, the asymptotic (in $n$) probability of a Condercet winner 
in the three-candidate mean-field model is $\frac{3}{2\pi} \arccos\left(\frac{3}{4\beta - 9}\right)$.
\end{theorem}
\begin{proof}
This follows from Lemma~\ref{lem:gaussian-sgn} and Lemma~\ref{lem:clt-magnetization} since the above expression equals $(3/4)(1 - \E[sgn(Y^1)sgn(Y^2)])$.
\end{proof}
This result gives the first part of the curve in Figure~\ref{fig:guilbaud}.
\subsection{Supercritical regime}
At a high level, the analysis of the supercritical regime proceeds in a similar way to the analysis of other mean-field models such as Curie-Weiss. That is to say, we can show fairly directly that the naive mean-field approximation is an accurate estimate of $\log Z$ (see the proof of Theorem~\ref{thm:low-temperature-soln}), so that the asymptotic value of $\frac{1}{n} \log Z$ is given by a concrete optimization problem over measures on $NAE_3$ (maximizing $\Phi(s)$ defined below). From a large deviation perspective, we know that for $X \sim Uni(NAE_3)$, the quantity $\|\frac{1}{n} \sum_i X_i\|$ converges to zero by the law of large numbers, and this optimization problem is asking for the typical behavior of $\frac{1}{n} \sum_i X_i$ if we condition on the unlikely event that $\|\frac{1}{n} \sum_i X_i\| = \Omega(1)$ --- this is made precise by Cramer's Theorem \cite{ellis2007entropy,dembo2011large}. 
The next step is to solve this optimization problem, which in our case is quite involved; once we have solved this problem, we can use a symmetry argument to characterize the limiting law of $\frac{1}{n} \sum_{i = 1}^n X_i$ and prove that the probability of Condorcet paradox is exponentially small. 

To begin the analysis, we recall a (slight special case of) the Gibbs variational principle \eqref{eqn:gibbs-principle}, when we look at $X$ distributed according to the uniform measure on a set $\mathcal{X}$.
\begin{lemma}[Gibbs variational principle, \cite{ellis2007entropy}]\label{lem:gibbs-variational}
Let $\mathcal{X}$ be a finite set. Then
\[ \log \sum_{x \in \mathcal{X}} e^{f(x)} = \sup_{P \in \mathcal{P}(\mathcal{X})} \E_P[f(X)] + H_P(X) \]
where $P$ ranges over all probability distributions of random variable $X$ valued in $\mathcal{X}$
and $H_P(X) = \E_P[-\log P(X)]$ is the Shannon entropy.
\end{lemma}
As explained above, $\Phi(s)$ is a functional such that the maximum of $\Phi$ corresponds to the value of the free energy $\frac{1}{n} \log Z$ (we prove this in Theorem~\ref{thm:low-temperature-soln}).
This raises the question of characterizing the maximum of $\Phi$. Because the inversion tilt model is a maximum-entropy distribution (from the above variational principle), we can derive mean-field equations by introducing dual variables $\lambda_1,\lambda_2,\lambda_3$, which are the parameters of an inversion tilt model, and relating them to primal variables $s_1,s_2,s_3$ at a critical point:
\begin{lemma}\label{lem:mf-eqn}
Fix $\beta$ and define 
\[ \Phi(s) := \frac{\beta}{2} \|s\|_2^2 + \max_{Q \in \mathcal{P}(NAE_3) : \E_Q X = s} H_Q(X). \]
where we interpret the maximum to be $-\infty$ if no such $Q$ exists.
Then at any critical point of $\Phi$,
\[ \Phi(s) = -\frac{\beta}{2} \|s\|_2^2 + \log 2 +  \log (\cosh(\lambda_1 + \lambda_2 - \lambda_3) + \cosh(\lambda_1 - \lambda_2 + \lambda_3) + \cosh(-\lambda_1 + \lambda_2 + \lambda_3)) \]
for $\lambda_1,\lambda_2,\lambda_3$ which are solutions of
\begin{align}\label{eqn:mf-eqn-1}
s_1 &= \frac{\sinh(\lambda_1 + \lambda_2 - \lambda_3) + \sinh(\lambda_1 - \lambda_2 + \lambda_3) - \sinh(-\lambda_1 + \lambda_2 + \lambda_3)}{\cosh(\lambda_1 + \lambda_2 - \lambda_3) + \cosh(\lambda_1 - \lambda_2 + \lambda_3) + \cosh(-\lambda_1 + \lambda_2 + \lambda_3)} \\
s_2 &= \frac{\sinh(\lambda_1 + \lambda_2 - \lambda_3) - \sinh(\lambda_1 - \lambda_2 + \lambda_3) + \sinh(-\lambda_1 + \lambda_2 + \lambda_3)}{\cosh(\lambda_1 + \lambda_2 - \lambda_3) + \cosh(\lambda_1 - \lambda_2 + \lambda_3) + \cosh(-\lambda_1 + \lambda_2 + \lambda_3)} \label{eqn:mf-eqn-2} \\
s_3 &= \frac{-\sinh(\lambda_1 + \lambda_2 - \lambda_3) + \sinh(\lambda_1 - \lambda_2 + \lambda_3) + \sinh(-\lambda_1 + \lambda_2 + \lambda_3)}{\cosh(\lambda_1 + \lambda_2 - \lambda_3) + \cosh(\lambda_1 - \lambda_2 + \lambda_3) + \cosh(-\lambda_1 + \lambda_2 + \lambda_3)} \label{eqn:mf-eqn-3}.
\end{align}
and also satisfy $\lambda_i = \beta s_i$ for all $i$.
\end{lemma}
\begin{proof}
Observe by the minimax theorem \cite{sion1958general} and the Gibbs variational principle (Lemma~\ref{lem:gibbs-variational}) that
\begin{align*}
\max_{Q \in \mathcal{P}(NAE_3) : \E_Q X = s} H_Q(X) 
&= \max_{Q \in \mathcal{P}(NAE_3)} \min_{\lambda} H_Q(X) + \langle \lambda, \E_Q X - s \rangle \\
&= \min_{\lambda} \max_{Q \in \mathcal{P}(NAE_3)} (H_Q(X) + \langle \lambda, \E_Q X - s \rangle) \\
&= \min_{\lambda} -\langle \lambda, s \rangle + \log \sum_{x \in NAE_3} e^{\langle \lambda, x \rangle}. 
\end{align*}
Note that for arbitrary $\lambda$ (not necessarily related to $s$), we have
\[ \log \sum_{x \in NAE_3} e^{\langle \lambda, x \rangle} = \log 2 + \log (\cosh(\lambda_1 + \lambda_2 - \lambda_3) + \cosh(\lambda_1 - \lambda_2 + \lambda_3) + \cosh(-\lambda_1 + \lambda_2 + \lambda_3)) \]
so using that $\E_{\lambda}[X] = \nabla_{\lambda} \log \sum_{x \in NAE_3} e^{\langle \lambda, x \rangle}$ and considering the first-order optimality conditions, we see that the optimizer $\lambda$ in the minimization problem above must satisfy \eqref{eqn:mf-eqn-1}, \eqref{eqn:mf-eqn-2}, and \eqref{eqn:mf-eqn-3} above.
It remains to show that at a critical point, $\lambda$ and $s$ satisfy $\lambda_i = \beta s_i$ for all $i$.
Write
\begin{align*} 
\Psi(\lambda) 
&:= \log 2 + \frac{\beta}{2} \|s(\lambda)\|_2^2 - \langle \lambda, s(\lambda) \rangle \\
&\quad + \log(\cosh(\lambda_1 + \lambda_2 - \lambda_3) + \cosh(\lambda_1 - \lambda_2 + \lambda_3) + \cosh(-\lambda_1 + \lambda_2 + \lambda_3)) 
\end{align*}
so that $\Phi(s) = \Psi(\lambda)$.
Observe that
\[ \Psi'(\lambda) = \beta s(\lambda)^T s'(\lambda) - s(\lambda)^T - \lambda^T s'(\lambda) + s(\lambda)^T = \beta s(\lambda)^T s'(\lambda) - \lambda^T s'(\lambda). \]
Furthermore, we claim that the Jacobian $s'(\lambda)$ is always invertible,
so the solutions must satisfy $\lambda = \beta s(\lambda)$. To see that the Jacobian is invertible, observe that $s'(\lambda) = \nabla^2 \log \sum_{x \in \NAE_3} e^{\langle \lambda, x \rangle} = (\E_{\lambda}[X_a X_b] - \E_{\lambda}[X_a]\E_{\lambda}[X_b])_{a,b} = \Sigma_{\lambda}$ where $\Sigma_{\lambda}$ is the covariance matrix of $X = \varphi(\pi)$ where $\pi$ is drawn from the Inversion Tilt Model with parameter $\lambda$. For any finite $\lambda$ we know for $a,b,c$ distinct elements of $\{1,2,3\}$ that $\Var(X_a | X_b,X_c) > 0$ which shows there cannot be any linear relation between the coordinates of random vector $X \in \{\pm 1\}^3$, hence $\Sigma_{\lambda}$ is invertible.
\end{proof} 
In order to understand the low-temperature behavior of the model, we will need to solve the mean-field equations from Lemma~\ref{lem:mf-eqn} which we do in the following Lemma~\ref{lem:mean-field-solns}. In the proof of Lemma~\ref{lem:mean-field-solns}, we solve the equations to the point where they reduce to concrete statements about analytic functions in one and two dimensions and then check those statements using some computer-generated plots (Figure~\ref{fig:u1u2solns} and Figure~\ref{fig:soln_qualities}). In principle the use of plots can be replaced by appeal to a formal decision procedure for the corresponding theory of real arithmetic, assuming a weak version of Schanuel's conjecture from field theory \cite{wilkie1995decidability}.
For completeness, in Appendix~\ref{apdx:mf-3} we also sketch the proof of a weaker version of Lemma~\ref{lem:mean-field-solns} which still suffices to prove Theorem~\ref{thm:low-temperature-soln} below, and doesn't rely upon the computer-generated plots.

\begin{lemma}\label{lem:mean-field-solns}
For all $\beta$, the solutions to the equations in Lemma~\ref{lem:mf-eqn} are of one of the following types, up to symmetries of permuting coordinates and $\lambda \mapsto -\lambda$:
\begin{enumerate}
    \item Of the form $\lambda_1 = \lambda_2 = \lambda_3$
    \item Of the form $\lambda_1 = 0, \lambda_2 = -\lambda_3$.
    \item Of the form $\lambda_1 = \lambda_2$ where $\lambda_3$ has the opposite sign of $\lambda_1$ and up to symmetries, this point is unique (for $\beta > 3/4$ it has an orbit of size exactly 6).
\end{enumerate}
Furthermore, the points of the third type are the global maximizers of $\Phi$ for every fixed value of $\beta$.
\end{lemma}

\begin{theorem}\label{thm:low-temperature-soln}
Fix arbitrary $\beta > 3/4$. Define $S_n := \frac{1}{\sqrt{n}} Y_n = \frac{1}{n} \sum_i X^{(n)}_i$ where $X^{(n)} \sim Q_n$ and $Q = Q_n$ is the mean-field model \eqref{eqn:mf-nae} on $n$ voters at inverse temperature $\beta$. In the limit as $n \to \infty$, the random variables $S_n$ converge in distribution to $S$ distributed
according to the uniform measure on the 6 global maximizers described in Lemma~\ref{lem:mean-field-solns}. Furthermore, if $\mathcal{S}_{\beta}$ is the set of these 6 points in the support of $S$, then for any $\epsilon > 0$ there exists $c = c_{\beta,\epsilon} > 0$ such that $Q(\min_{s \in \mathcal{S}_{\beta}} \|S_n - s\| \ge \epsilon) \le 2 e^{-c n}$.
\end{theorem}
\begin{proof}
 Observe that the support of $S_n$ is contained in the compact set $[-1,1]^3$. Therefore, by Prokhorov's theorem \cite{billingsley2013convergence} there exists at least one subsequential limit $\mu$ of the sequence of measures of $S_1,S_2,S_3,\ldots$ and we proceed to characterize this limit.

Observe that
$Q(S_n = s) \propto \exp\left(n \frac{\beta}{2} \|s\|_2^2 + \log \#\{x \in \NAE_3^n : s = \sum_i x\}\right)$.
Therefore for any measurable set $K$,
$Q(S_n \in K) \propto \sum_{s \in U} \exp\left(n \frac{\beta}{2} \|s\|_2^2 + \log \#\{x \in \NAE_3^n : s = \sum_i x\}\right)$.
If $K$ is a closed set that does not contain one of the 6 maximizers of the mean-field problem described in Lemma~\ref{lem:mean-field-solns}, then 
we will show by applying large deviations methods
that $\lim_{n \to \infty} Q(S_n \in K) = 0$. Explicitly, if we let $\epsilon > 0$ be such that $\max_{s \in K} \Phi(s) < \Phi(s^*) - \epsilon$ where $s^*$ is a maximizer of $\Phi$ over all $s$, where such $\epsilon$ exists by compactness of $K$, then by the union bound over the possible values of $Y_n$ which lie in $K$ we get that 
\begin{align*}
    Q(S_n \in K) 
    &= \frac{1}{Z_n} \sum_{s_n \in K} \exp\left(n \frac{\beta}{2} \|s\|_2^2 + \log \#\{x \in \NAE_3^n : s = \sum_i x\}\right) \\
    &\le \frac{C}{Z_n} n^3 \max_{s_n \in K} \exp\left(n \frac{\beta}{2} \|s\|_2^2 + \log \#\{x \in \NAE_3^n : s = \sum_i x\}\right) \\ 
    &\le \frac{C}{Z_n} n^3 \max_{s_n \in K} \exp\left(n \frac{\beta}{2} \|s\|_2^2 + n H(X_1 | S_n = s_n)\right)
    \le \frac{C n^3 \exp(n (\Phi(s^*) - \epsilon))}{Z_n}.
\end{align*}
where in the first inequality we used that there are only $O(n^3)$ possible values for $S_n$, in the second inequality we used that $\log |S| = H(X)$ where $X$ is chosen uniformly at random over $S$ \cite{cover2012elements}, the sub-additivity inequality $H(X_1,\ldots,X_n) \le \sum_i H(X_i)$ for entropy \cite{cover2012elements}, and the fact that the coordinates of $X$ are all symmetrical, and in the final inequality used the definitions of $\epsilon$ and $\Phi$. 
We know by the Gibbs variational principle (Lemma~\ref{lem:gibbs-variational}) that by plugging into the supremum the product measure $P_{\lambda}^{\otimes n}$ where $P_{\lambda}$ is the Inversion Tilt Model with $\lambda = s^* \beta$ that $\log Z_n \ge n\Phi(s^*)$ therefore we see that the above probability converges to $0$ as $n \to \infty$. Furthermore, this proves the large deviation bound in the statement of the Theorem by taking $K$ to be the set of points with distance at least $\epsilon$ from $\mathcal{S}_{\beta}$.

It follows that any limiting distribution must be supported on the set of 6 symmetrical global maximizers from Lemma~\ref{lem:mean-field-solns}. Since they are symmetrical and the random variables $S_n$ respect this symmetry, the measure must be the uniform measure. It follows that this is the unique limiting distribution.
\end{proof}
\begin{remark}
Interestingly, this result tells us that society breaks symmetry in a way such that all pairwise elections are won by a $\Omega(n)$ margin of votes (which would not happen if, for example, the solutions of type 2 in Lemma~\ref{lem:mean-field-solns} were optimal).
\end{remark}
\begin{corollary}\label{corr:no-paradox}
For $\beta > 3/4$, the asymptotic (as $n \to \infty$) probability of a Condorcet winner in the three-candidate mean-field model under pairwise majority is equal to $1$. Furthermore, there exists $c = c_{\beta} > 0$ such that the probability is at least $1 - 2 \exp(-c_{\beta} n)$.
\end{corollary}
\begin{proof}
This follows from Theorem~\ref{thm:low-temperature-soln} and analysis of the six symmetrical cases: in particular, if $\lambda_1 = \lambda_2 > 0$ and $\lambda_3 < 0$, then under pairwise majority election with probability $1 - o_{n \to \infty}(1)$, candidate $1$ will beat candidates $2$ and $3$ in the corresponding pairwise elections, and $2$ will beat $3$, so candidate 1 is a Condorcet winner. The existence of $c_{\beta} > 0$ follows from the above case analysis and the large deviation bound in Theorem~\ref{thm:low-temperature-soln}.
\end{proof}
\section{Mean-Field Model with Arbitrarily Many Candidates}\label{sec:large-q}
Based on the analysis in the $q = 3$ setting, a plausible conjecture would be that many of the same behaviors occur in the case $q > 3$: if $\beta_c$ corresponds to the ``critical temperature'' of the model (for us, the point where the limit of the free energy $\frac{1}{n} \log Z(\beta)$ as $n \to \infty$ is not analytic, which we expect to be unique) then for the entire high-temperature phase $\beta < \beta_c$ the Quantitative Arrow's Theorem holds, and in the entire low-temperature phase $\beta_c < \beta$ elections by pairwise majority should yield a Condorcet winner with probability $1 - o_{n \to \infty}(1)$. 

It turns out generalizing these results to larger $q$ is difficult. 
One key step in the proof of Theorem~\ref{thm:qa-mf-3} fails to generalize to larger values of $q$: in Lemma~\ref{lem:w_n-ui}, the upper bound on the cumulant generating function $\log \E[\exp(\lambda \langle Y_n, u \rangle)]$ by Taylor expansion and naively comparing monomials term by term, i.e. Lemma~\ref{lem:cosh-inequality}, gives a very weak bound for larger values of $q$ (the constant grows exponentially with $q$). Therefore new ideas are needed to prove an upper bound on this constant (the sub-Gaussian constant of the random vector $\varphi(\pi)$) which scales correctly with $q$. 

In the next section, we prove a large deviation bound which is optimal up to a constant factor of at most $3$ (for all $q$), from which the Quantitative Arrow's Theorem for a high-temperature regime $\beta \le 1/(q - 1)$ is derived. In the final section we show the model enters a low-temperature phase for $\beta > 3/(q + 1)$ by lower bounding $\frac{1}{n} \log Z$, ruling out the possibility that $\frac{1}{n} \log Z$ is an analytic extension of its high temperature behavior (i.e. constant) which means the critical temperature must be in-between. We also state a simple conjecture which would eliminate the gap of $3$ between these two regimes, identifying the location of the phase transition.
\subsection{Large deviations inequality}
As described above, the key tool we need to extend the $q = 3$ high-temperature analysis to $q > 3$ is a good estimate on the cumulant generating function of $\varphi(\pi)$ for $\pi \sim \mathfrak{S}_q$. The following Lemma gives a sub-Gaussian inequality, proved via martingale and symmetrization techniques, which for every $q$ is sharp up to a constant factor of at most 3; however we conjecture below (Conjecture~\ref{conj:large-q-conjecture}) that the sharp constant for this inequality is better, generalizing the bound we proved earlier for the $q = 3$ case in a natural way.

Crucially, Lemma~\ref{lem:general-large-deviations} is much stronger than the trivial estimate one gets from the fact $\|X\|_2 = O(q)$, which by Hoeffding's inequality \cite{vershynin2018high} implies a bound for the sub-Gaussian constant of the incorrect order $O(q^2)$. 
\begin{lemma}\label{lem:general-large-deviations}
For any vector $\lambda$, if if we let random vector $X = \varphi(\pi)$ for $\pi \sim Uniform(\mathfrak{S}_q)$ with $q \ge 2$ then
\[ \log \E[\exp(\langle \lambda, X \rangle)] \le \frac{q - 1}{2}\|\lambda\|_2^2. \]
\end{lemma}
\begin{proof}
Observe by pairing up each permutation with its reversed version (i.e. the permutation $\pi' = \pi \circ (i \mapsto q - i + 1)$) that $\E[\exp(\langle \lambda, X \rangle)] = \E[\cosh(\langle \lambda, X \rangle)]$. 

We now prove the inequality using a Doob martingale. For the filtration, we think of the permutation as being generated by a Fisher-Yates shuffle, i.e. picking in order $j_1 = \pi^{-1}(1), j_2 = \pi^{-1}(2), \ldots$ where at each step $j_t$ is chosen uniformly at random from the elements of $[q] \setminus \{j_1,\ldots,j_{t - 1}\}$. Define $\lambda_{j_1}$ by projection onto the set of coordinates which involve $j_1$ (i.e. indexed by pairs $(i,j_1)$ and $(j_1,k)$) and define $\lambda_{\sim j_1}$ to be the complement, and define $X_{j_1},X_{\sim j_1}$ likewise. Using the identity $\cosh(a + b) = \cosh(a)\cosh(b) + \sinh(a)\sinh(b)$, observe that
\begin{align*} 
\E[\cosh(\langle \lambda, X \rangle)] 
&= \E[\cosh(\langle \lambda_{j_1}, X_{j_1} \rangle)\cosh(\langle \lambda_{\sim j_1}, X_{\sim j_1} \rangle) + \sinh(\langle \lambda_{j_1}, X_{j_1} \rangle)\sinh(\langle \lambda_{\sim j_1}, X_{\sim j_1}  \rangle)] \\
&= \E[\cosh(\langle \lambda_{j_1}, X_{j_1} \rangle)\cosh(\langle \lambda_{\sim j_1}, X_{\sim j_1} \rangle)] \\
&= \E[\cosh(\langle \lambda_{j_1}, X_{j_1} \rangle)\E[\cosh(\langle \lambda_{\sim j_1}, X_{\sim j_1} \rangle) | j_1]] \\
&= \E[\cosh(\langle \lambda_{j_1}, X_{j_1} \rangle) \cdots \cosh(\langle \lambda_{j_{q - 1}}, X_{j_{q - 1}} \rangle)]
\end{align*}
where in the second equality we used that $\E[\sinh(\langle \lambda_{\sim j_1}, X_{\sim j_1} \rangle) | j_1] = 0$ by symmetry considerations, and in the last equality we applied the argument inductively on $\E[\cosh(\langle \lambda_{\sim j_1}, X_{\sim j_1} \rangle) | j_1]$.
Using the inequality $\cosh(x) \le e^{x^2/2}$, the Cauchy-Schwartz inequality, and Holder's inequality we see that
\begin{align*} 
\cosh(\langle \lambda_{j_1}, X_{j_1} \rangle) \cdots \cosh(\langle \lambda_{j_{q - 1}}, X_{j_{q - 1}} \rangle) 
&\le \exp\left(\sum_t \langle \lambda_{j_t}, X_{j_t} \rangle^2/2\right) \\
&\le \exp\left(\sum_t \|\lambda_{j_t}\|^2 \frac{q - 1 - t}{2}\right)
\le \exp\left(\|\lambda\|^2 \frac{q - 1}{2}\right)
\end{align*}
so taking the expectation and logarithm gives the result.
\end{proof}
\begin{remark}
The replacement of $\exp$ by $\cosh$ in the proof of Lemma~\ref{lem:general-large-deviations} is equivalent to ``symmetrizing'' the Fisher-Yates shuffle: at every step of the shuffle, we flip a fair coin and depending on its outcome inserts a randomly chosen element either at the front of the list or at the back of the final list, then iterate recursively on the remaining elements.
\end{remark}
\begin{conjecture}\label{conj:large-q-conjecture}
In the same setting as Lemma~\ref{lem:general-large-deviations}, we conjecture that
\[ \log \E[\exp(\langle \lambda, X \rangle)] \le \frac{q + 1}{6} \|\lambda\|_2^2. \]
If true, this constant is sharp because it is attained in the limit $\lambda \to 0$ (by Taylor expansion and computing the top eigenvalue of the covariance matrix, see Lemma~\ref{lem:cov-eigenvalues}).
\end{conjecture}
\begin{remark}
In terms of concentration inequalities, the sub-Gaussian bound in Lemma~\ref{lem:general-large-deviations} gives that for $X = \varphi(\pi)$, $\pi \sim Uni(\mathfrak{S}_q)$, and $\|w\| \le 1$
\[ Q[\langle w, X \rangle > t] \le e^{-t^2/2(q - 1)} \]
where $Q$ is the law of $X$,
and Conjecture~\ref{conj:large-q-conjecture} would give the improved estimate
\[ Q[\langle w, X \rangle > t] \le e^{-3t^2/2(q + 1)} \]
sharpening the constant in the exponent. We remark that in the special case that $w$ is along the all-ones direction, the quantity $\langle w, X \rangle$ correspond to Kendall's $\tau$ statistic and its concentration plays an important role in statistical tests (see e.g. \cite{hoeffding1948class}).
\end{remark}
Some preliminary computer simulations support this conjecture for small values of $q$, and were also performed in \cite{raffaelli2005statistical}. As we will see in the next sections, if we knew this conjecture then we could identify
the critical temperature for these models exactly as well as prove a Quantitative Arrow's Theorem in the entire high-temperature regime.
\subsection{High-temperature Quantitative Arrow's Theorem}
From the large deviation inequality above (Lemma~\ref{lem:general-large-deviations}), we can use similar arguments to the proof of Theorem~\ref{thm:qa-mf-3-intro} to prove a Quantitative Arrow's Theorem under the assumption $\beta < 1/(q - 1)$. The details of the proof are left to the Appendix.
\subsection{Low-temperature behavior}
In this section, we prove the model must be in its low-temperature phase for $\beta > 3/(q + 1)$ in the sense it is not mutually contiguous to the product measure, and the limiting behavior of $\frac{1}{n} \log Z$ is not an analytic extension of its high temperature behavior (i.e. it is not equal to a constant); therefore the model must exhibit a phase transition at or before $3/(q + 1)$ and the high-temperature contiguity estimate from the previous section is off by at most a factor of (slightly less than) three. The techniques are based upon the naive mean-field approximation and computing its second order expansion around its ``trivial'' critical point at the origin; details and proofs are left to the Appendix.
\section{Perfect Matching Model}\label{sec:matching}
We proceed to study the perfect matching model \ref{eqn:matching-intro}
mentioned in the introduction. We recall that we denote the pairs of voters in a matching by $X_i$ and $Y_i$ and let $n$ denote the total number of matchings (i.e. half the number of total voters), so that the joint measure will be
\begin{equation} 
Q(X = x, Y = y) = \frac{1}{Z} \exp\left(\beta \sum_{i = 1}^n \langle \varphi(x_{i}), \varphi(y_{i}) \rangle\right).
\label{eqn:matching-model}
\end{equation}


In the first subsection below, we compute the probability of a Condorcet winner under pairwise elections.
Surprisingly, our calculation shows that one feature of the mean-field model, that the probability of a Condorcet winner is increasing in $\beta$, is \emph{not} universal. Instead, the probability is increasing only for small values of $\beta$; for larger values of $\beta$, increasing the strength of interactions in the model \emph{monotonically decreases} the probability of a Condorcet winner (see Figure~\ref{fig:matching-behavior}).  

We then proceed to prove the main result of this section, that the Quantitative Arrow's Theorem holds for every $\beta \ge 0$. The proof is technically involved: we first give a high level overview of the proof, state and prove some needed estimates, and then show step-by-step how to adapt the proof from the original product measure setting \cite{mossel2012quantitative}. 
\subsection{Probability of paradox under pairwise majority}
Using the Central Limit Theorem, we can derive the analogous version of Guilbaud's formula for the matching by computing the covariance matrix corresponding to each pair of voters. We leave the calculation to the Appendix.
\begin{theorem}[Restatement of Theorem~\ref{thm:matching-guilbaud}]\label{thm:matching-guilbaud}
For $\beta > 0$, let $(X,Y)$ be distributed according to the Gibbs measure $Q_n$ \eqref{eqn:matching-model} with $2n$ voters.
Let $E_n$ be the event that there is a Condorcet winner under pairwise majority elections. Then
\[ \lim_{n \to \infty} Q_n(E_n) = \frac{3}{2\pi} \arccos \left(\frac{-1/3 - \frac{\sinh(3\beta) + 2\sinh(\beta)}{3(\cosh(3\beta) + 2\cosh(\beta))}}{ 1 + \frac{3\sinh(3\beta) + 2\sinh(\beta)}{3(\cosh(3\beta) + 2\cosh(\beta))}}\right). \]
\end{theorem}

\subsection{Quantitative Arrow's Theorem}
In this section, for all finite values of $\beta \ge 0$ we prove that given $\epsilon > 0$ the probability of paradox for constitutions $\epsilon$-far from $\mathcal{F}_3$ is lower bounded by $\delta = \delta(\beta,\epsilon)$ independent of $n$ for any constitution satisfying the hypotheses of Arrow's theorem. In this setting, the result cannot be proved by a mutual contiguity argument: there are constitutions with a low probability of paradox in the matching model which have a high probability of paradox in the product measure setting; this is consistent with the Quantitative Arrow's Theorem because the notion of $\epsilon$-close is distribution dependent. For example, constitutions which count the number of agreements between $X_{a1}$ and $Y_{a1}$ for $a = 1$ to $n$ will behave very differently in the product measure ($\beta = 0$) and finite temperature ($\beta > 0$) setting.

\subsubsection{Proof Strategy} 
The proof strategy follows the same general template as the proof for independent voters developed in \cite{mossel2012quantitative} and uses general reverse hypercontractive estimates developed in \cite{mossel2013reverse}. The key similarity that allows us to use these techniques is that the distribution of votes is still given by a product measure if we group pairs of voters; the main obstacle that we have to overcome is that unlike the setting with truly independent voters, the conditional law of the election between candidates $2$ and $3$ given the results of an election between candidates $1$ and $2$ is significantly more complicated. In particular, this complicates the step of the argument where low-influence functions are handled using the Invariance Principle and arguments in Gaussian space. In the first two subsections below, we develop the necessary estimates needed to overcome this issue using reverse hypercontractivity and linear algebraic tools such as Schur complement formulae.

For the benefit of the readers who are not familiar with~\cite{mossel2012quantitative}, we describe the main steps and the modifications needed for proof here at a high level. 

\begin{enumerate}
\item[I.] First a Gaussian version of the Theorem is formulated and proved. One advantage of Gaussian space is that it has no dictators, and therefore, the statement is simpler: that unless some choices are almost fixed, there is a good probability of paradox. The covariance structure of the Gaussian random variables should match those of the binary random variable. In the case of i.i.d. voters, this is a $3$ dimensional structure representing the correlation between different choices of the same voter. In our case, choices of adjacent voters are correlated. Thus we need to work with a $9$-dimensional covariance structure.\footnote{The factor of $3$ for the 3 different candidates is the same as before, and another factor of $3$ appears because the preferences from a pair of adjacent voters between two candidates corresponds to an element of $\{\pm 1\}^2$, the space of functions from $\{\pm 1\}^2 \to \mathbb{R}$ is four-dimensional, and one dimension corresponds to constant functions.}
However, it turns our that the details of the covariance structure do not matter much, as long as 
Reverse Hyper-Contractive inequalities hold (see III. below). 

\item[II.] Once a Gaussian version is proven, and using the Majority is Stablest Theorem, one can deduce the same statement as long as all of the influences are small. Here we prove a similar statement involving the influences of each pair of voters.

\item[III.] Using the Reverse Hyper-Contractive inequality by Borell~\cite{Borell:82}, in \cite{mossel2012quantitative} it was shown that if two voters $i,j$ have high influence for two different functions $f,g$, then the probability of a paradox is high.
We prove a similar statement here, though we need to apply a more general Hyper-Contractive inequality 
from~\cite{mossel2013reverse}. Moreover, we require a number of estimates to show that we can apply such inequalities both in the discrete and in the Gaussian setup (for Step I.), which are given in Section~\ref{subsubsec:voter-freedom}.

\item[IV.] The remaining case is where is only one voter that is influential. In this case, by conditioning on the vote of this voter and applying the low influence result in II., it is possible to conclude that the function is either close to a dictator or has a high probability of paradox. We use a similar argument, where some additional work is needed to get from two voters to one voter. 
\end{enumerate}

\subsubsection{Reverse hypercontractive estimates}\label{subsubsec:reverse-hypercontractive}
First we recall an important estimate for discrete distributions which follows
from a general form of reverse hypercontractivity.
\begin{lemma}[Lemma 8.3 of \cite{mossel2013reverse}]\label{lem:reverse-discrete}
Suppose $\Omega$ is a finite probability space and $(X_i,Y_i)_{i = 1}^n$ are i.i.d. jointly distributed $\Omega^2$-valued random variables; let $P$ be their joint law.
Suppose
\[ \alpha := \min_{a,b : P(Y_1 = b) > 0} \frac{P(X_1 = a, Y_1 = b)}{P(X_1 = a)P(Y_1 = b)} > 0. \]
Let $X = (X_1,\ldots,X_n)$ and $Y = (Y_1,\ldots,Y_n)$.
Then for any sets $A,B \subset \Omega^n$ such that $P(X \in A) \ge \epsilon$ and $P(Y \in B) \ge \epsilon$ it follows that
\[ P(X \in A, Y \in B) \ge \epsilon^{\frac{2 - \sqrt{1 - \alpha}}{1 - \sqrt{1 - \alpha}}}. \]
\end{lemma}
We also need an analogous estimate for correlated Gaussian vectors, which we will prove below as Lemma~\ref{lem:corr-reverse}. To state the result, we need an important Schur complement formula for multivariate Gaussians (see e.g. \cite{zhang2006schur} for a reference).
\begin{lemma}[Schur Complement Formula \cite{zhang2006schur}]
Suppose $X$ and $Y$ are zero-mean correlated Gaussian vectors with joint covariance matrix
\[ \Sigma = \begin{bmatrix} 
\Sigma_{X} & \Sigma_{XY} \\
\Sigma_{YX} & \Sigma_{Y}
\end{bmatrix}. \]
Then the law of $Y$ conditional on $X$ is given by 
$N\left(\Sigma_{YX} \Sigma_{XX}^{-1} X, \Sigma_Y - \Sigma_{YX} \Sigma_{XX}^{-1} \Sigma_{XY}\right)$
\end{lemma}
In particular, note that the conditional covariance matrix does not depend on the particular value of $X$. We also need the reverse hypercontractive estimate for Gaussians established in \cite{Borell:82}. Here as usual $\|f\|_p := \E[f^p]^{1/p}$, and we use this notation even when $p < 1$ (where $\| \cdot \|_p$ is no longer a norm).
\begin{defn}[Ornstein-Uhlenbeck semigroup]
The Ornstein-Uhlenbeck semigroup $T_t$ is defined for $t \ge 0$ as
\[ (T_t f)(x) := \E_{Z \sim N(0,1)}[f(e^{-t} x + \sqrt{1 - e^{-2t}} Z)] \]
for $f \in L^1(\gamma)$ and $\gamma$ the standard Gaussian measure $N(0,1)$.
\end{defn}
\begin{theorem}[Reverse Hypercontractivity for Gaussians \cite{Borell:82}]\label{thm:borell-gaussian}
For any strictly positive $f$ and $p,q < 1$
\[ \|T_t f\|_{L^q(\gamma)} \ge \|f\|_{L^p(\gamma)} \]
as long as $t \ge \frac{1}{2}\log \frac{1 - q}{1 - p}$,
where $\gamma$ is the standard Gaussian measure.
\end{theorem}
We can now give the needed reverse hypercontractive estimates for general correlated Gaussian vectors:
\begin{lemma}\label{lem:corr-reverse}
Suppose that $(X_i,Y_i)_{i = 1}^n$ are i.i.d. samples from $N(0,\Sigma)$ with $\Sigma \succ 0$ and block decompose 
\[ \Sigma = \begin{bmatrix} \Sigma_X & \Sigma_{XY} \\ \Sigma_{YX} & \Sigma_Y \end{bmatrix}. \]
Let $P$ be the joint law of $(X_i,Y_i)_{i = 1}^n$.
Define
\[ \alpha := \frac{1}{\lambda_{max}(\Sigma_{Y}^{-1/2} \Sigma_{YX} \Sigma_{X}^{-1} \Sigma_{XY} \Sigma_{Y}^{-1/2})} - 1 \]
and suppose $\alpha > 0$ (this is implied by $\Sigma_{Y | X} \succ 0$).
Let $X = (X_1,\ldots,X_n)$ and $Y = (Y_1,\ldots,Y_n)$.
Then:
\begin{enumerate}
    \item 
    Let $K$ be the Markov kernel corresponding to the conditional law of $Y_1 | X_1$, then for any strictly positive $f$ and arbitrary $q < p < 1$,
    \[ \|(K^{\otimes n} f)(X)\|_{L^q(\mu^{\otimes n})} \ge \|f(Y)\|_{L^p(\nu^{\otimes n})} \]
    where $\mu$ is the (Gaussian) law of $X_1$, $\nu$ is the law of $Y_1$,
    and $t \ge \frac{1}{2} \log \frac{1 - q}{1 - p}$ where
    \[ t := \frac{1}{2} \log (1 + \alpha). \]
    \item For all $0 < p,q < 1$ and all nonnegative $f,g$,
    \[ \E[f K^{\otimes n} g] \ge \|f\|_{L^q(\mu^{\otimes n})} \|g\|_{L^p(\nu^{\otimes n})} \]
    where $K,\mu,\nu,t$ are defined as above and assuming $t \ge \frac{1}{2} \log \frac{1}{(1 - p)(1 - q)}$.
    \item For any measurable sets $A,B \subset \mathbb{R}^{nd}$ such that $P(X \in A) \ge \epsilon$ and $P(Y \in B) \ge \epsilon$ it follows that
\[ P(X \in A, Y \in B) \ge \epsilon^{\frac{2 - e^{-t/2}}{1 - e^{-t/2}}} \]
where $t$ is as defined above.
\end{enumerate}
\end{lemma}
\begin{proof}
First we can reduce to the case $\Sigma_X = I, \Sigma_Y = I$ by defining $X' = \Sigma_X^{-1/2} X$ and $Y' = \Sigma_Y^{-1/2} Y$; observe that $\Sigma_{X'} := \E[(X')(X')^T] = I$, similarly $\Sigma_{Y'} = I$, and $\Sigma_{X' Y'} := \E[X' (Y')^T] = \Sigma_X^{-1/2} \Sigma_{XY} \Sigma_Y^{-1/2}$, so
\[ \Sigma_{Y' X'} \Sigma_{X' Y'} = \Sigma_Y^{-1/2} \Sigma_{YX} \Sigma_{XX}^{-1} \Sigma_{XY} \Sigma_{YY}^{-1/2} \]
which shows that moving to $X'$ preserves the value of $\alpha$, and all other quantities (functional norms, etc.) are clearly basis invariant. From now on we assume $\Sigma_X = I, \Sigma_Y = I$.

(1). For (1), this form of reverse hypercontractive inequality is known to tensorize (see \cite{Borell:82,mossel2013reverse}) so it suffices to prove the result in the case $n = 1$.
Explicitly, 
we show that $K$ factorizes so that $K(f) = T_{t}(S(f))$ for $t > 0$ defined in the theorem statement and $S$ a Markov kernel, and hence
\[ \|K f\|_{L^q(\mu)} = \|T_{t} (S f)\|_{L^q(\mu)} \ge \|S f\|_{L^p(\mu)} \ge \|f\|_{L^p(\nu)} \]
where the first inequality is by the tensorized version of Theorem~\ref{thm:borell-gaussian}, and the last inequality follows from pointwise application of Jensen's inequality, as in the proof of Lemma 8.1 in \cite{mossel2013reverse}: if $p \in (0,1)$ then $(S f)^p \ge S(f^p)$ and if $p < 0$ then $(S f)^p \le S(f^p)$. 

In order to derive the factorization, we define
$\xi := Y - \E[Y | X]$ and observe that $\xi \sim N(0, \Sigma_{Y | X})$ and $\xi$ is independent of $X$. Using the formula for $\E[Y | X]$ in Gaussians and $\Sigma_{X} = I_{d \times d}$ gives
\[ Y = \Sigma_{YX} X + \xi. \]
Also observe the following decomposition holds (a form of law of total variance)
\[ I_{d \times d} = \E[Y Y^T] = \Sigma_{YX} \Sigma_{XY} + \Sigma_{Y | X} \]
so $\Sigma_{Y | X} = I_{d \times d} - \Sigma_{YX} \Sigma_{XY}$. 
Therefore conditional on $X$, we have the equality in law
\[ Y \stackrel{d}{=} \Sigma_{YX} (X + \sqrt{\alpha} \xi_1) + \xi_2\]
where $\xi_1 \sim N(0,I)$, $\xi_2 \sim N(0, \Sigma_{Y | X} - \alpha \Sigma_{YX} \Sigma_{XY})$. Note that $\Sigma_{Y | X} - \alpha \Sigma_{YX} \Sigma_{XY} = I - (1 + \alpha) \Sigma_{YX} \Sigma_{XY} \succeq 0$ by the definition of $\alpha$, so this is valid. Therefore conditional on $X$ we have the equality in law
\[ Y \stackrel{d}{=} \Sigma_{YX}\sqrt{1 + \alpha} (\sqrt{1 - \alpha} X + \sqrt{\frac{\alpha}{1 + \alpha}} \xi_1) + \xi_2 = \Sigma_{YX} \sqrt{1 + \alpha}(e^{-t} X + \sqrt{1 - e^{-2t}} \xi_1) + \xi_2 
\]
since $t$ is defined so that $1 - e^{-2t} = \frac{\alpha}{1 + \alpha}$. Defining $X_t := e^{-t} X + \sqrt{1 - e^{-2t}} \xi_1$ we have from the above decomposition of the law of $Y$ given $X$ that
\[ X \to X_t \to Y \]
is a Markov chain. The Markov kernel $S$ in our desired factorization then is given by the conditional law of $Y$ given $X_t$, since the conditional law of $X_t$ given $X$ matches the Ornstein-Uhlenbeck process as desired.

(2). By the reverse Holder's inequality (\cite{Borell:82}, Lemma 5.2 of \cite{mossel2013reverse}),
\[ \E[f K^{\otimes n} g] \ge \|f\|_{L^q} \|K^{\otimes n} g\|_{L^{q'}}\]
where $q' := \frac{q}{q - 1}$ so that $\frac{1}{q} + \frac{1}{q'} = 1$. Then the result follows by part (1) as long as
\[ t \ge \frac{1}{2} \log \frac{1 - q'}{1 - p} = \frac{1}{2} \log \frac{1}{(1 - p)(1 - q)}. \]

(3). This follows from (2) in the same way as the proof of Lemma 8.3 of \cite{mossel2013reverse}.
\end{proof}
\subsubsection{Voters are not determined by their neighbors.}\label{subsubsec:voter-freedom}
In order to apply the reverse hypercontractive estimates from the previous section, we need to show that for any $\beta$ and $(X,Y)$ a pair of neighboring voters that they have a positive probability of choosing any one of the voting outcomes in $\NAE_3 \times \NAE_3$. 
\begin{lemma}\label{lem:voter-freedom-1}
Fix $\beta \ge 0$. Suppose that $(X,Y)$ are random vectors jointly valued in $\NAE_3 \times \NAE_3$ and distributed according to the Gibbs measure
\[ Q(X = x, Y = y) = \frac{1}{Z} \exp\left(\beta \langle x, y \rangle\right) \]
where $Z$ is the normalizing constant.
Then
\[ \alpha := \min_{x,y \in \NAE_3} \frac{Q(X = x, Y = y)}{Q(X = x)Q(Y = y)} > 0. \]
\end{lemma}
\begin{proof}
Since probabilities are upper bounded by 1, we have
\[ \alpha \ge \min_{x,y} Q(X = x, Y = y) = \min_{x,y} \frac{\exp\left(\beta \langle x, y \rangle\right)}{\sum_{x',y'}\exp\left(\beta \langle x, y \rangle\right)} \ge \frac{\exp(-3 \beta)}{36 \exp(3 \beta)} = e^{-6\beta}/36 \]
using that $|\langle x, y\rangle| \le 3$ and $|\NAE_3 \times \NAE_3| = 36$.
\end{proof}
When applying the Invariance Principle using results from \cite{mossel2010gaussian} we will need a bound on the maximum correlation between the first coordinates of $X_1,Y_1$ and the second coordinates of $X_2,Y_2$. We adopt a multi-index notation where $X_{11}$ is the first coordinate of $X_1$, $X_{12}$ is the second coordinate, etc.
\begin{lemma}\label{lem:max-correlation}
Let $\beta,X,Y$ be as defined in Lemma~\ref{lem:voter-freedom-1} and let $i \ne j$ be elements of $\{1,2,3\}$.
Define the maximum correlation coefficient
\[ \rho_{ij}(\beta) := \sup_{f,g : \{\pm 1\}^2 \to \mathbb{R}} \frac{\Cov(f(X_{1i},Y_{1i}), g(X_{1j},Y_{1j}))}{\sqrt{\Var(f(X_{1i},Y_{1i})) \Var(g(X_{1j},Y_{1j}))}}. \]
For all $\beta \ge 0$, $\rho_{ij}(\beta) < 1$.
\end{lemma}
\begin{proof}
Without loss of generality $i = 1$ and $j = 2$. By rescaling, we can restrict to $f,g$ satisfying $\Var(f) = \Var(g) = 1$. After the rescaling, the set of functions we are optimizing over is compact so it suffices to show that for any particular choice of $f$ and $g$ that $\Cov(f,g) < 1$. Suppose otherwise, so $\Cov(f,g) = 1$, then $g(X_{12},Y_{12})$ is a deterministic affine function of $f(X_{11},Y_{11})$. However, from the assumption $\Var(g) = 1$ we know that there are two inputs to $g$ such that it takes on different values, and from Lemma~\ref{lem:voter-freedom-1} we know that both of those inputs have positive probability of occurring regardless of the value of $X_{11},Y_{11}$, so $\Var(g(X_{12},Y_{12}) | X_{11},Y_{11}) > 0$. By contradiction, $\rho_{12}(\beta) < 1$.
\end{proof}
We also prove a strong nondegeneracy property of the joint covariance matrix which we will need when we apply reverse hypercontractivity to moment-matched Gaussians coming from the Invariance Principle. Informally, it expresses the fact that although there exist nontrivial functions like $X \mapsto \bone[(X_{11},X_{12},X_{13}) \in \NAE_3]$ which are constant under the Gibbs measure, they are not linear functions of $V$ defined in Lemma~\ref{lem:voter-freedom-2} (or multilinear functions of the vector version defined in Lemma~\ref{lem:gaussian-version}), which will mean that those functions essentially do not exist when we move to the Gaussian version of the problem.
\begin{lemma}\label{lem:voter-freedom-2}
Fix $\beta \ge 0$ and let $(X,Y)$ be as defined in Lemma~\ref{lem:voter-freedom-1} so that $X$ and $Y$ are correlated random vectors both valued in $\NAE_3$.
Define $\Sigma : 9 \times 9$ to be the covariance matrix of the random vector $V = (V_1,V_2,V_3)$ where
\[ V_i := \left( X_{i}, Y_{i}, \frac{X_{i}Y_{i} - \mu}{\sigma} \right)_{i = 1}^3 \]
and where, by symmetry, $\mu = \E[X_{i}Y_{i}]$ and $\sigma = \sqrt{\Var(X_{i}Y_{i})}$ do not depend on $i$. Then $\Sigma$ is positive definite, i.e. $\Sigma \succ 0$.
\end{lemma}
\begin{proof}
Just for this proof, we adopt the multi-index notation that $V_{ij} = (V_i)_j$.
Let $w$ be an arbitrary vector in $\mathbb{R}^9$ with $\|w\|_2 = 1$ and indexed in the same way as $V$. To prove the Theorem, it suffices to show that for every such $w$, $\Var(\sum_{i,j} w_{ij} V_{ij}) > 0$ as this implies the minimum eigenvalue of $\Sigma$ is positive. There must exist $a \in \{1,2,3\}$ and $b \in \{1,2,3\}$ be such that $w_{ab} \ne 0$ and without loss of generality assume that $a = 1$. From the law of total variance, we have
\[ \Var(\sum_{i,j} w_{ij} V_{ij}) \ge \E \Var(\sum_j w_{1j} V_{1j} | V_2, V_3) = \E \Var(\sum_j w_{1j} V_{1j} | X_2, X_3, Y_2, Y_3). \]
Since all values of $X,Y$ in $\NAE_3^2$ occur with positive probability by Lemma~\ref{lem:voter-freedom-1}, to show the above is positive it suffices to show that $\Var(\sum_j w_{1j} V_{1j} | X_2 = 1, X_3 = -1, Y_2 = 1, Y_3 = -1) > 0$. In this case the conditional law of $X_1,Y_1$ is such that $\E X_1 = \E Y_1 = 0$ and $|\E[X_1 Y_1]| < 1$, so the functions $\{1, X_1, Y_1, X_1 Y_1\}$ form a linearly independent basis for the space of functions under this measure, hence $\{X_1,Y_1, \frac{X_1 Y_1 - \mu}{\sigma}\}$ must also be linearly independent: otherwise, we could solve for $X_1$ in terms of the other three basis elements.
\end{proof}
\subsubsection{Proof of Quantitative Arrow's Theorem}
Equipped with the estimates derived in the previous two subsections and a general formulation of the Invariance Principle \cite{mossel2005noise,mossel2010gaussian}, it now becomes relatively straightforward to adapt the proof of Quantitative Arrrow's Theorem for independent voters developed in \cite{mossel2012quantitative,mossel2013reverse}. The steps in this argument were sketched at the beginning of this section (i.e. steps I-IV) and the argument follows the same order. Essentially, compared to the proof for independent voters we need to: 1. be careful about replacing voters by pairs of voters in various parts of the proof: for example, in the discussion of influence and in the appeal to the Invariance principle. 2. replace the appeal to various estimates by their replacement given in the previous two subsections, and 3. at the end of the argument, appeal to (the non-quantitative version of) Arrow's Theorem to show that a constitution which is close to a junta of two paired voters is either close to an actual dictator or has a significant probability of paradox.

\sparagraph{Gaussian Version}
We start with the proof of the Gaussian analogue. In this analogue, the correlated random vectors $V(1),V(2),V(3)$ defined below contain information about the three pairwise elections between candidates. Since the Gaussian analogue corresponds to the setting with low-influence functions (see next subsection), the dictator function is no longer a way to avoid paradox. For technical reasons involving the reduction via the Invariance Principle (and as in \cite{mossel2012quantitative,mossel2013reverse}), this result needs to be proved for functions valued in $[-1,1]$.
\begin{lemma}\label{lem:gaussian-version} 
Fix $\beta > 0$ and define $\Sigma : 9 \times 9$ as in Lemma~\ref{lem:voter-freedom-2}.
Suppose that $V_1,\ldots,V_n \sim N(0, \Sigma)$ where each $V_i$ further decomposes as $V_i = (V_{i1}, V_{i2}, V_{i3})$ as in Lemma~\ref{lem:voter-freedom-2}. For $a \in \{1,2,3\}$, define $V(a) = (V_{ia})_{i = 1}^n$. 
Fix $\epsilon > 0$.
There exists $\delta = \delta(\epsilon) > 0$ such that at least one of the following occurs, for any $f,g,h$ which are measurable functions from $\mathbb{R}^{3n} \to [-1,1]$:
\begin{enumerate}
    \item Two of the random variables $f(V(1)), g(V(2)), h(V(3))$ are $\epsilon$-close to constant functions of the opposite sign. Here we say $f$ is $\epsilon$-close to the constant function $1$ if $\E[f] \ge 1 - 2\epsilon$.
    \item The probability of paradox is lower bounded by $\delta$: \[ \E[\text{nae}_3(f(V(1)), g(V(2)), h(V(3)))] > 1 - \delta \]
    where $\text{nae}_3 : [-1,1]^3 \to [0,1]$ is defined to be the harmonic extension of the indicator function for the set $\NAE_3$, explicitly
    $\text{nae}_3(x,y,z) = \frac{1}{4}\left(3 - xy - yz - xz\right)$.
\end{enumerate}
\end{lemma}
\begin{proof}
This follows from the proof of Theorem 11.7 in \cite{mossel2012quantitative}, except that we use the reverse hypercontractive estimate from Lemma~\ref{lem:corr-reverse} instead of Lemma 2.5 of \cite{mossel2012quantitative}; the condition $\alpha > 0$ is satisfied by Lemma~\ref{lem:voter-freedom-2} and the fact that the Schur complement of a positive definite matrix is always positive definite. 
\end{proof}

\sparagraph{Low (Cross-) Influence Functions.} Next, we derive Arrow's Theorem for low influence functions using the Invariance Principle to reduce to the Gaussian case. 
First we state the needed invariance principle.
\begin{defn}
Suppose $X_1,\ldots,X_n$ are i.i.d. random variables each valued in finite set $\mathfrak{X}$ with law $P$. Let $e_0,\ldots,e_k$ be an orthonormal basis for $L^2(P)$ with $e_0 = 1$ and for a multi-index $\sigma$ define $e_{\sigma} = \prod_i e_{i \sigma_i}$ in the usual way.
The \emph{$d$-low-degree influence} of coordinate $i$ on function $f$ is defined by
\[ I_i^{\le d}(f) := \sum_{\sigma : |\sigma| \le d, \sigma_i > 0} \E[e_{\sigma}(X) f(X)]^2.  \]
The \emph{influence} of coordinate $i$ is $I_i(f) := I_i^{\le n}(f)$. 
\end{defn}
\begin{lemma}\label{lem:invariance}
Let $\epsilon > 0$ and $i \ne j$ two distinct elements of $\{1,2,3\}$. Fix $\beta \ge 0$ and let $(X_1,Y_1),\ldots,(X_n,Y_n)$ be distributed according to the Gibbs measure \eqref{eqn:matching-model}. There exists $\tau := \tau(\epsilon,\beta) > 0$ such that the following is true.

Let $(N_1,M_1),\ldots,(N_n,M_n)$ be i.i.d. jointly Gaussian random variables drawn from $N(0,\Sigma)$ where $\Sigma : 6 \times 6$ is the covariance matrix of the random vector
\[ \left( X_{1i}, Y_{1i}, \frac{X_{1i}Y_{1i} - \mu}{\sigma}, X_{1j}, Y_{1j}, \frac{X_{1j}Y_{1j} - \mu}{\sigma} \right) \]
where $\mu = \E[X_{1i}Y_{1i}] = \E[X_{1j} Y_{1j}]$, $\sigma = \sqrt{\Var(X_{1i}Y_{1i})} = \sqrt{\Var(X_{1j}Y_{1j})}$,
and with $N_1$ corresponding to the first three coordinates and $M_1$ the last three.
Let $f : \{\pm 1\}^{2n} \to [-1,1]$ and $g : \{\pm 1\}^{2n} \to [-1,1]$ be arbitrary functions such that for all $1 \le a \le n$,
\[ \max\left\{I_a^{\le \log(1/\tau)}(f), I_a^{\le \log(1/\tau)}(g)\right\} < \tau \]
where $I_a^{\le d}(f)$ is the $d$-low degree influence of coordinate $a \in \{1,\ldots,n\}$ on $f$ viewed as a function of $(X_{ai},Y_{ai})_{i = 1}^n$, and defined analogously for $g$.
Then there exist functions $\tilde{f},\tilde{g} : \mathbb{R}^{3n} \to [-1,1]$ such that
\[ |\E[f(X(i), Y(i)) g(X(j), Y(j))] - \E[\tilde{f}(N) \tilde{g}(M)]| \le \epsilon \]
where the notation $X(i)$ denotes the vector $(X_{ai})_{a = 1}^n$, and where $\tilde{f}$ is defined only in terms of $f$ (i.e. it is independent of the choice of $g$).
\end{lemma}
\begin{proof}
This follows from the argument of Theorem 11.9 of \cite{mossel2012quantitative} (see also Lemma A.4 of \cite{mossel2013reverse}) by combining Lemma~\ref{lem:max-correlation}, Lemma 6.1 of \cite{mossel2010gaussian}, and Theorem 3.20 of \cite{mossel2005noise} using hypothesis H3 there.
\end{proof}

\begin{lemma}\label{lem:low-influence}
For every $\beta \ge 0$ and $\epsilon > 0$, there exists $\delta(\epsilon,\beta), \tau(\epsilon,\beta) > 0$ such that the following result holds, where $X,Y$ are sampled from the Gibbs measure $Q$ \eqref{eqn:matching-model}. 
Let $f,g,h : \{\pm 1\}^{2n} \to \{\pm 1\}$ be arbitrary. Suppose that for all $1 \le a \le n$, at most one of $f,g,h$ satisfies $I_a^{\log^2(1/\tau)} > \tau$ where we view $f$ as a function of the independent pairs $(X_{a1},Y_{a1})_{a = 1}^n$ and similarly for $g$ and $h$. Then either:
\begin{enumerate}
    \item Two of $f,g,h$ are $\frac{3}{2} \epsilon$-close to constant functions of opposite sign.
    \item The probability of paradox is lower bounded by $\delta$, i.e. \[ Q((f(X^1,Y^1),g(X^2,Y^2),h(X^3,Y^3)) \in \NAE_3) \le 1 - \delta. \]
\end{enumerate}
\end{lemma}
\begin{proof}
The proof follows Theorem 11.11 of \cite{mossel2012quantitative}, except that the version of invariance we use is Lemma~\ref{lem:invariance} and the proof in the Gaussian case is Lemma~\ref{lem:gaussian-version}. 
\end{proof}

\sparagraph{Constitution with Two Influential Pairs.} In the case that two pairs of voters are both influential, there is always a significant probability of paradox.
\begin{lemma}\label{lem:two-influential}
For every $\beta \ge 0$ and $\epsilon > 0$, there exists $\delta(\epsilon,\beta)$ such that the following result holds, where $X,Y$ are sampled from the Gibbs measure $Q$ \eqref{eqn:matching-model}. As in the previous Lemmas, we consider arbitrary functions $f,g,h : \{\pm 1\}^{2n}$ and view them as functions of independent pairs of coordinates.
Suppose there exist two indices $1 \le a < b \le n$ and $f_1 \ne f_2$ are two distinct elements of $\{f,g,h\}$ such that  $I_a(f_1) > \epsilon$ and $I_b(f_2) > \epsilon$. Then the probability of paradox is lower bounded: $Q((f(X^1,Y^1),g(X^2,Y^2),h(X^3,Y^3)) \in \NAE_3) > 1 - \delta$.
\end{lemma}
\begin{proof}
The proof follows Theorem 3.3 of \cite{mossel2012quantitative}, except that we look at pairs of voters instead of single voters and use the general reverse hypercontractive estimate from Lemma~\ref{lem:reverse-discrete}, in the application of which we can justify that $\alpha > 0$ by appealing to Lemma~\ref{lem:voter-freedom-1}.
\end{proof}

\sparagraph{Constitutions with One Influential Pair.} In the case that only one pair in the matching has large influence, we can condition on this pair to generate a situation where no pairs are influential. This implies by the previous Lemmas that the pair is a dictator, from which it is easy to show that a single element in the pair is the dictator. 

\begin{lemma}\label{lem:one-influential}
For every $\beta \ge 0$ and $\epsilon > 0$, there exists $\delta(\epsilon,\beta), \tau(\epsilon,\beta) > 0$ such that the following result holds, where $X,Y$ are sampled from the Gibbs measure $Q$ \eqref{eqn:matching-model}.  As in the previous Lemmas, we consider arbitrary functions $f,g,h : \{\pm 1\}^{2n}$ and view them as functions of independent pairs of coordinates. Assume that there exists $1 \le a \le n$ such that for all $b \ne a$,
\[ \max(I_b(f), I_b(g), I_b(h)) < \tau. \]
Then either:
\begin{enumerate}
    \item the function $(X,Y) \mapsto (f(X^1,Y^1),g(X^2,Y^2),h(X^3,Y^3))$ is $\epsilon$-close to a function in $\mathcal{F}_3$, i.e. the constitution is $\epsilon$-close to either being dictator or having a fixed bottom or top candidate.
    \item Or, the probability of paradox is lower-bounded, i.e. 
    \[ Q((f(X^1,Y^1), g(X^2,Y^2), h(X^3,Y^3)) \in \NAE_3) \ge \delta. \]
\end{enumerate}
\end{lemma}
\begin{proof}
First, we prove the result with $\epsilon$ replaced by $\epsilon/4$ and where in case (1) we expand the definition so that the constitution is allowed to depend nontrivially on both voters in the matching (i.e. those two voters form a junta). In this case, the proof follows as in Theorem 7.1 of \cite{mossel2012quantitative}, except that we use Lemma~\ref{lem:low-influence} to handle the low-influence case after conditioning in the argument.

Finally, if the constitution is close to constitution depending on only the two voters in the matching indexed by $a$, the standard Arrow's Theorem (see Proposition 3.1 of \cite{mossel2012quantitative}) and the triangle inequality can be applied to show that it either has a significant probability of paradox (so it falls into case (2)), or it is $\epsilon$-close to an element of $\mathcal{F}_3$.
\end{proof}

\sparagraph{Deduction of Quantitative Arrow's Theorem.}
\begin{theorem}
For every $\beta \ge 0$ and $\epsilon > 0$, there exists $\delta(\epsilon,\beta) > 0$ such that the following result holds, where $X,Y$ are sampled from the Gibbs measure $Q$ \eqref{eqn:matching-model}.  As in the previous Lemmas, we consider arbitrary functions $f,g,h : \{\pm 1\}^{2n}$ and view them as functions of independent pairs of coordinates. 
Then either:
\begin{enumerate}
    \item the function $(X,Y) \mapsto (f(X^1,Y^1),g(X^2,Y^2),h(X^3,Y^3))$ is $\epsilon$-close to a function in $\mathcal{F}_3$, i.e. the constitution is $\epsilon$-close to either being dictator or having a fixed bottom or top candidate.
    \item Or, the probability of paradox is lower-bounded, i.e. 
    \[ Q((f(X^1,Y^1), g(X^2,Y^2), h(X^3,Y^3)) \in \NAE_3) \ge \delta. \]
\end{enumerate}
\end{theorem}
\begin{proof}
This is by case analysis as in Theorem 11.14 of \cite{mossel2012quantitative}: all constitutions fall into the setting of one of Lemma~\ref{lem:one-influential}, Lemma~\ref{lem:two-influential}, or Lemma~\ref{lem:low-influence}; here we use that low-degree influences always lower bound normal influences.
\end{proof}

\section{Conclusion}
Beyond the conjectures already stated, a number of interesting open problems remain. We state a few natural questions below:
\begin{enumerate}
    \item Supposing that Conjecture~\ref{conj:universal-arrow} is true, it's also interesting to ask for each $q$ about the supremum of $\beta(q)$ such that the result holds; it seems very plausible that the sharp $\beta(q)$ in general is determined by the mean-field case where $G$ is the complete graph. In some sense this would show that the complete graph is the ``best case'' for avoiding paradox.
    \item 
    In the cases considered in this paper, pairwise majority-based elections were shown to be threshold-optimal in the models considered with $q = 3$, in the sense that whenever Quantitative Arrow's Theorem does not hold (i.e. Arrow's paradox is avoidable asymptotically almost surely), pairwise majority also avoids a paradox asymptotically almost surely. This is consistent with the following much more general hypothesis: for any $q \ge 3$, $\gamma > 0$, there exists a function $\delta(q,\epsilon,\gamma)$ such that if elections under pairwise majority exhibit at least an $\gamma$ probability of paradox, the Quantitative Arrow's Theorem also holds -- either the constitution is $\epsilon$-close to a function in $\mathcal{F}_k$ or the probability of a paradox is at least $\delta(q,\epsilon,\gamma) > 0$. This hypothesis implies that no other voting rule (which is far from dictator and far from constant in every pairwise election) succeeds in avoiding a paradox when pairwise majority fails. Is the hypothesis true?  
    \item Conjecture~\ref{conj:universal-arrow} asks for the behavior on general graphs, but it is also interesting to understand the sharp regime for a Quantitative Arrow's Theorem to hold on particular families of graphs (e.g. lattices). This is closely related to the previous two questions.    
    \item In the low-temperature case ($\beta$ large) and for general $q$, is it true that for all connected graphs, pairwise majority avoids a paradox asymptotically almost surely? If so, is the probability of paradox always exponentially small in $n$?
    \item What more can be said about the probability of a paradox as a function of $\beta$ --- for example, how does it behave if we zoom in to the critical temperature?
    In the case of the perfect matching, is the minimum probability of paradox still attained at the same point if instead of pairwise majority, we consider the optimal voting rule for each value of $\beta$? More generally, when is the probability of paradox monotone in $\beta$?
\end{enumerate}
\appendix
\section{Appendix: Deferred Proofs from Section~\ref{sec:mf-3}}\label{apdx:mf-3}
\subsection{Proof of Lemma~\ref{lem:mean-field-solns}}
\begin{proof}
By using $\lambda_i = \beta s_i$ and eliminating $\beta$, we see that any solution $(\lambda_1,\lambda_2,\lambda_3)$ to the mean-field equations (Lemma~\ref{lem:mf-eqn}) for some value of $\beta$
must satisfy the three equations
\begin{align*}
\lambda_1 &(\sinh(\lambda_1 + \lambda_2 - \lambda_3) - \sinh(\lambda_1 - \lambda_2 + \lambda_3) + \sinh(-\lambda_1 + \lambda_2 + \lambda_3)) \\
&= \lambda_2 (\sinh(\lambda_1 + \lambda_2 - \lambda_3) + \sinh(\lambda_1 - \lambda_2 + \lambda_3) - \sinh(-\lambda_1 + \lambda_2 + \lambda_3)) \\
\lambda_1 &(-\sinh(\lambda_1 + \lambda_2 - \lambda_3) + \sinh(\lambda_1 - \lambda_2 + \lambda_3) + \sinh(-\lambda_1 + \lambda_2 + \lambda_3)) \\
&= \lambda_3(\sinh(\lambda_1 + \lambda_2 - \lambda_3) + \sinh(\lambda_1 - \lambda_2 + \lambda_3) - \sinh(-\lambda_1 + \lambda_2 + \lambda_3)) \\
\lambda_2&(-\sinh(\lambda_1 + \lambda_2 - \lambda_3) + \sinh(\lambda_1 - \lambda_2 + \lambda_3) + \sinh(-\lambda_1 + \lambda_2 + \lambda_3) \\
&= \lambda_3(\sinh(\lambda_1 + \lambda_2 - \lambda_3) - \sinh(\lambda_1 - \lambda_2 + \lambda_3) + \sinh(-\lambda_1 + \lambda_2 + \lambda_3)).
\end{align*}
We make the change of variables
\begin{align*}
    u &:= \begin{bmatrix} 
    1 & 1 & -1 \\
    1 & -1 & 1 \\
    -1 & 1 & 1
    \end{bmatrix} \lambda.
\end{align*}
Note that this change of variables preserves the symmetry of the equations under permutation of $u_1,u_2,u_3$. 
This allows us to focus mostly on the first two equations above
\begin{align*}
    \frac{u_1 + u_2}{2} (\sinh(u_1) - \sinh(u_2) + \sinh(u_3)) &= \frac{u_1 + u_3}{2} (\sinh(u_1) + \sinh(u_2) - \sinh(u_3)) \\
    \frac{u_1 + u_2}{2} (-\sinh(u_1) + \sinh(u_2) + \sinh(u_3)) &= \frac{u_2 + u_3}{2} (\sinh(u_1) + \sinh(u_2) - \sinh(u_3))
\end{align*}
and use the third equation when arguing by symmetry. 
These equations are equivalent to their sum and difference, which are:
\begin{align}
 (u_1 + u_2)\sinh(u_3) &= \frac{u_1 + u_2 + 2u_3}{2} (\sinh(u_1) + \sinh(u_2) - \sinh(u_3)) \label{eqn:summed}\\
    (u_1 + u_2)(\sinh(u_1) - \sinh(u_2)) &= \frac{u_1 - u_2}{2}(\sinh(u_1) + \sinh(u_2) - \sinh(u_3)). \label{eqn:differenced}
\end{align}
We now break into two cases:
\begin{itemize}
    \item Case 1: $u_1 = u_2$. Then \eqref{eqn:summed} gives
    \begin{equation}\label{eqn:u1u3}
    2u_1 \sinh(u_3) = (u_1 + u_3) (2\sinh(u_1) - \sinh(u_3)). \end{equation}
    We plot the solution locus in Figure~\ref{fig:u1u3solns}.
    Rearranging, the equation is $2u_1 \sinh(u_1) + 2u_3\sinh(u_1) = 3\sinh(u_3) u_1 + u_3 \sinh(u_3)$. For a fixed value of $u_3$, since the left hand side is strictly convex in $u_1$ and the right hand side is linear, there exist at most two solutions to the equation. Furthermore for $u_3 \ne 0$ the left hand side is zero and the right hand side is positive, so there are exactly two solutions. One family of solutions is given by $u_1 = u_3$, and then the other family of solutions has $u_1$ with the opposite sign of $u_3$.
    \item Case 2: $u_1 \ne u_2$. We split into further subcases.
    \begin{enumerate}
    \item $u_1 = u_3$. Then the symmetrical version of \eqref{eqn:summed} gives us a symmetrical version of \eqref{eqn:u1u3}: the argument above tells us the resulting equation has a single family of solutions where $u_2$ has opposite sign to $u_1 = u_3$. (There would be another family of solutions where $u_1 = u_2$ but in case 2 we have ruled out those solutions.)
    \item $u_2 = u_3$. This case is symmetrical to the previous case as well, giving a single family of solutions with $u_1$ having opposite sign to $u_2$ and $u_3$. 
    \item $u_2 \ne u_3$ and $u_1 \ne u_3$. Using that $u_1 \ne u_2$ and \eqref{eqn:differenced} lets us solve for $u_3$:
    \begin{equation}\label{eqn:u3-fromu1u2}
    \sinh(u_3) = \sinh(u_1) + \sinh(u_2)-\frac{2(u_1 + u_2)(\sinh(u_1) - \sinh(u_2))}{u_1 - u_2}.  
    \end{equation}
    Using this \eqref{eqn:summed} can be rewritten as
\begin{align}\label{eqn:summed-rewrite1}
    (u_1 + u_2)\sinh(u_3) = (u_1 + u_2 + 2u_3) \frac{(u_1 + u_2)(\sinh(u_1) - \sinh(u_2))}{u_1 - u_2}
\end{align}
We consider some further subcases:
\begin{enumerate}
    \item $u_1 + u_2 = 0$.  Then \eqref{eqn:u3-fromu1u2} lets us solve to get $u_3 = 0$.
    \item Two symmetrical cases to the previous one: $u_1 + u_3 = 0, u_2 + u_3 = 0$ which have symmetrical solution families.
    \item Finally, we have the case where $0 \notin \{u_1 + u_2, u_2 + u_3, u_1 + u_3\}$. In this case we will argue there is no solution. Dividing by $u_1 + u_2$ in \eqref{eqn:summed-rewrite1} gives
    \begin{align}\label{eqn:u1u2}
    \sinh(u_3) = (u_1 + u_2 + 2u_3) \frac{\sinh(u_1) - \sinh(u_2)}{u_1 - u_2}.
    \end{align}
\end{enumerate}
We plot the solution locus of the equation \eqref{eqn:u1u2} (with $u_3$ defined in terms of $u_1,u_2$ by \eqref{eqn:u3-fromu1u2}, and requiring $u_1 \ne u_2$) in Figure~\ref{fig:u1u2solns}. The five curves which appear are the five subcases of case 2 above (case 2.1, 2.2, 2.3.a and two in 2.3.b) so we indeed covered all of the cases. 
\end{enumerate}
\end{itemize}
Assuming the accuracy of Figure~\ref{fig:u1u2solns} and changing back into the original variables, we get the classification described in the first part of the Lemma. For the second part, it remains to plot the objective value achieved by each of the solution families as a function of $\beta$.
Assuming the accuracy of Figure~\ref{fig:soln_qualities}, we have that the solutions of the third type have the largest objective value. 
\end{proof}

\begin{figure}
    \centering
    \includegraphics[scale=0.3,trim=1in 3.5in 0.5in 1in, clip]{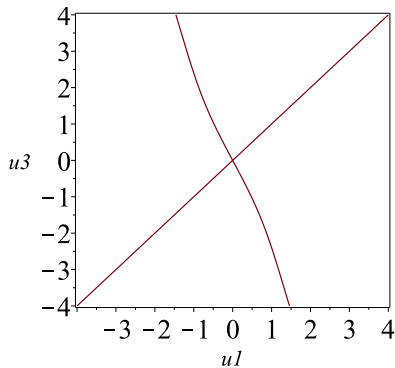}
    \caption{Solution set in terms of $u_1,u_3$ of \eqref{eqn:u1u3} 
    where $u_2 = u_1$.}
    \label{fig:u1u3solns}
\end{figure}
\begin{figure}
    \centering
    \includegraphics[scale=0.3,trim=1in 2in 0.5in 1in, clip]{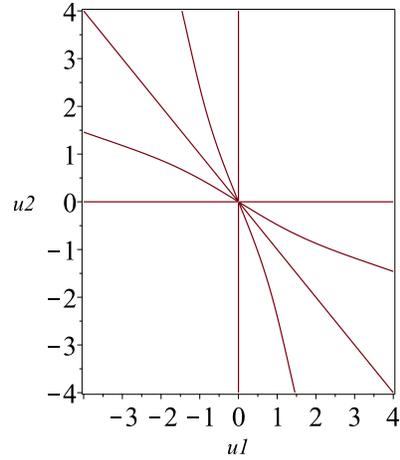}
    \caption{Solution set in terms of $u_1,u_2$ of \eqref{eqn:u1u2}
    where $u_3$ is determined by \eqref{eqn:u3-fromu1u2} and we require $u_1 \ne u_2$.}
    \label{fig:u1u2solns}
\end{figure}

\begin{figure}
    \centering
    \includegraphics[scale=0.20]{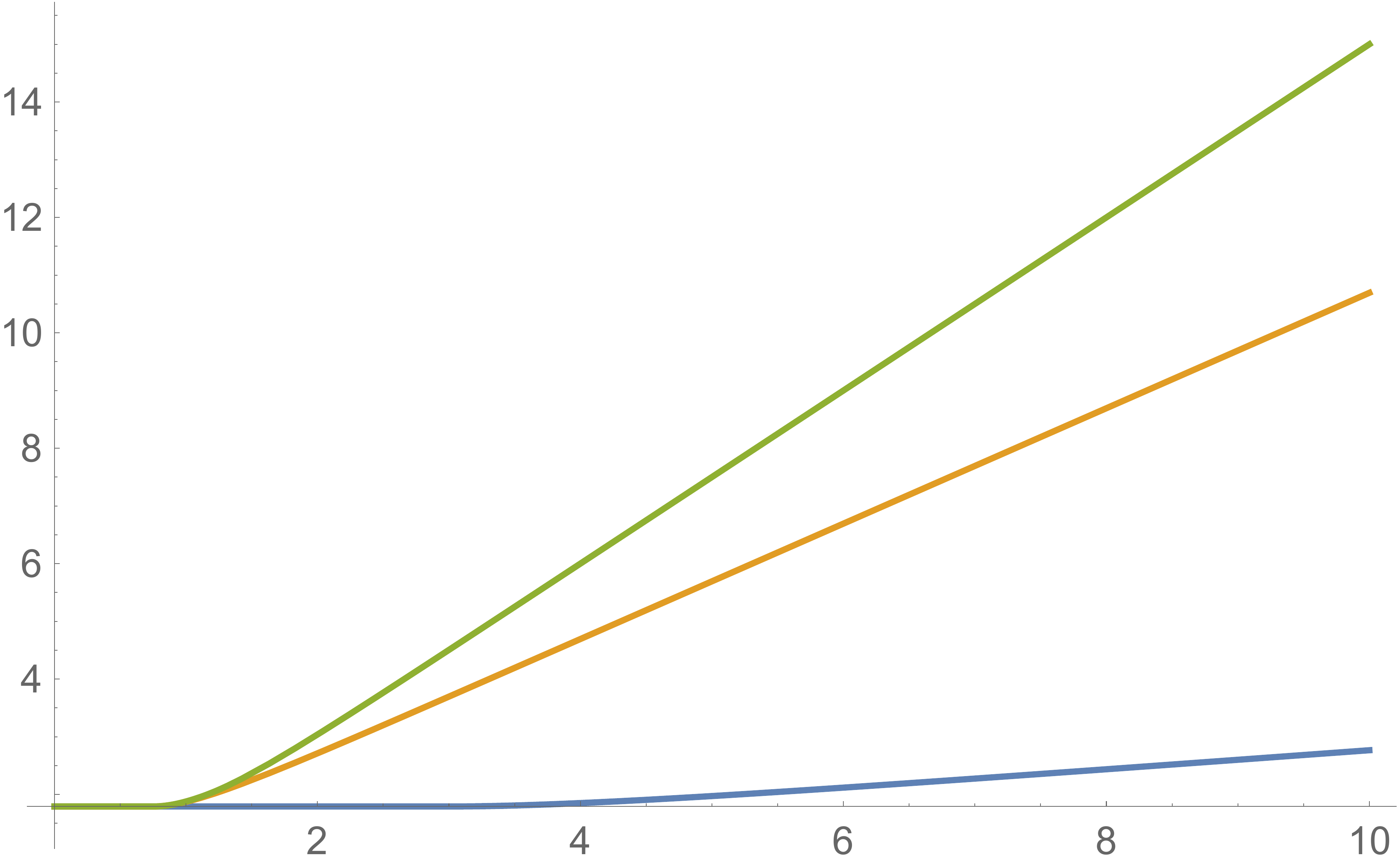}
    \caption{Plot of $\Phi(s)$ for varying $\beta$ and for each of the three types of solutions from Lemma~\ref{lem:mean-field-solns}: blue is type (1), orange is type (2), green is type (3).}
    \label{fig:soln_qualities}
\end{figure}

\sparagraph{Sketch of alternative analysis: } Here we sketch the proof of a weaker version of Lemma~\ref{lem:mean-field-solns} which avoids the use of the computer-generated plots and suffices for Theorem~\ref{thm:low-temperature-soln}. The weaker result we prove is that there are no solutions to the mean field equations from Lemma~\ref{lem:mf-eqn} except of the form $\lambda_1 = \lambda_2 = \lambda_3$ where all three of $\lambda_1,\lambda_2,\lambda_3$ have the same sign. We can check by plugging into the explicit solutions that for $\beta > 3/4$, the solutions of with all coordinates equal (i.e. type (1) in Lemma~\ref{lem:mf-eqn}) have smaller objective value than those of type (2) (or, of type (3) which are actually the global maximizer), and so no maximizer of the variational problem will not have all three of $\lambda_1,\lambda_2,\lambda_3$ with the same sign. Therefore (using the same argument in Theorem~\ref{thm:low-temperature-soln}) the law of $S_n$ will converge to a mixture of delta distributions supported on election results where the asymptotic probability of Condorcet paradox is zero. 

To prove the weaker result described above, we start from the remaining Case 2.3.c and \eqref{eqn:u3-fromu1u2} and \eqref{eqn:u1u2} derived previously.
Substituting with \eqref{eqn:u3-fromu1u2} on the left hand side of \eqref{eqn:u1u2} and rearranging gives
\begin{align*}
     \sinh(u_1) + \sinh(u_2)
     &= (3 u_1 + 3 u_2 + 2u_3) \frac{\sinh(u_1) - \sinh(u_2)}{u_1 - u_2}
\end{align*}
    so
    \[ u_3 = \frac{(\sinh(u_1) + \sinh(u_2))(u_1 - u_2)}{2(\sinh(u_1) - \sinh(u_2))} - \frac{3}{2}(u_1 + u_2). \]
    or equivalently
    \[ \frac{2u_3}{u_1 + u_2} = \frac{(\sinh(u_1) + \sinh(u_2))(u_1 - u_2)}{(\sinh(u_1) - \sinh(u_2))(u_1 + u_2)} - 3. \]
    Define $a = (u_1 + u_2)/2$ and $b = (u_1 - u_2)/2$ so $u_1 = a + b,u_2 = a - b$ then since $\sinh(x + y) = \sinh(x)\cosh(y) + \cosh(x)\sinh(y)$ the above is
    \[ \frac{u_3}{a} = \frac{\sinh(a)\cosh(b)b}{\cosh(a)\sinh(b)a} - 3 = \frac{\tanh(a) b}{\tanh(b) a} - 3\]
    i.e.
    \[ \frac{2u_3}{u_1 + u_2} = \frac{\tanh((u_1 + u_2)/2)}{\tanh((u_1 - u_2)/2)} \frac{u_1 - u_2}{u_1 + u_2} - 3. \]
    Define $h(x) = x/\tanh(x)$ which is a convex even function and note the above is
    \begin{equation}\label{eqn:weak-useful}
    \frac{2u_3}{u_1 + u_2} = \frac{h((u_1 - u_2)/2)}{h((u_1 + u_2)/2)} - 3
    \end{equation}
    Using symmetry we may assume that if there exists a solution, there exists
    one with $u_1 > u_2 \ge 0$ since two out of three numbers always have the same sign. In this case we see the right hand side is strictly smaller than $-2$ so $u_3 < -u_1 - u_2 < -2u_2$. Changing back to the original variables we have 
    \[ \lambda = \frac{1}{2} \begin{bmatrix} 
    1 & 1 & 0 \\
    1 & 0 & 1 \\
    0 & 1 & 1 \end{bmatrix} u
    \]
    and we see that $\lambda_1 > 0$ while $\lambda_2,\lambda_3 < 0$. Since the orbits of such a solution under $\pm$ symmetry and $\mathfrak{S}_3$ always have two coordinates of one sign and one coordinate of the opposite sign, this proves the weaker form of Lemma~\ref{lem:mean-field-solns}.
\begin{remark}
From \eqref{eqn:weak-useful} and its two symmetrical versions, we can also prove there are no other solutions to the mean field equations (from Lemma~\ref{lem:mf-eqn}) other than the claimed ones if we are given the following fact about a 1-parameter family of discrete time dynamical systems: that for all $u > 0$ the only points with orbits of periodicity 1 or 2 under the map
\[ g(\alpha) = \frac{\tanh(\frac{1 + \alpha}{2} u)}{\tanh(\frac{1 - \alpha}{2} u)} \frac{1 - \alpha}{2} - 3 \frac{1 + \alpha}{2}, \]
i.e. solutions to $(g \circ g)(\alpha) = \alpha$, are contained in $[-1,0]$; note it's easy to see that $g(-1) = 0$ and $g(0) = -1$. This fact also can be verified in principle by a decision theory for real arithmetic with exponentiation \cite{wilkie1995decidability}, and seems independently interesting. 
\end{remark}
\section{Deferred material from Section~\ref{sec:large-q}}
\subsection{Proof of Theorem~\ref{thm:quantitative-arrow-allq-intro}}
As before we will derive the Quantitative Arrow's Theorem from the product measure case and a contiguity argument. First we recall the statement in the product measure setting:
\begin{theorem}[Quantitative Arrow Theorem \cite{mossel2012quantitative}]\label{thm:quantitative-arrow-original}
Fix $q \ge 3$.
Suppose each voter votes independently uniformly at random from $\mathfrak{S}_q^{n}$. Fix $\epsilon > 0$. There exists $\delta = \delta(\epsilon,q) > 0$ such that for any constitution $F : \mathfrak{S}_q^n \to \{\pm 1\}^{q \choose 2}$ satisfying Independence of Irrelevant Alternatives (IIA), either:
\begin{enumerate}
    \item $F$ is $\epsilon$-close to a function in $\mathcal{F}_q$ with respect to the law of $X$; in particular, $F$ is close to being a dictator in some elections, or having some fixed pairwise elections.
    \item Or, the probability of paradox is lower bounded by $\delta$: if $X$ is the vector of votes drawn from the model \eqref{eqn:mean-field-intro}, the probability that the aggregated preference vector $F(X) \in \{\pm 1\}^{q \choose 2}$ satisfies transitivity is at most $1 - \delta$.
\end{enumerate}
\end{theorem}
\begin{proof}[Proof of Theorem~\ref{thm:quantitative-arrow-allq-intro}]
The proof follows the same strategy as in the $q = 3$ case, where we reduce to the known case of product measures using a contiguity estimate. Repeating the argument in the proof of Theorem~\ref{thm:qa-mf-3}, we will see that mutual contiguity holds for any $\beta$ such that we can prove for $X_i$ drawn i.i.d. from the uniform measure on the image of $\varphi$,
$Y_n := \frac{1}{\sqrt{n}} \sum_{i = 1}^n \varphi(X_i)$, and $W_n := \exp(\frac{\beta}{2} \langle Y_n, Y_n \rangle)$ that $W_n$ is uniformly integrable. Following the argument of Lemma~\ref{lem:w_n-ui}, we observe that by Lemma~\ref{lem:general-large-deviations}
\begin{align*} 
\max_{\|u\|_2 = 1} \log \E[\exp(\lambda \langle Y_n, u \rangle)]
= \max_{\|u\|_2 = 1} n \log \E_{\pi \sim \mathfrak{S}_q}[\exp(\frac{\lambda}{\sqrt{n}} \langle \varphi(\pi), u \rangle]
\le \frac{q - 1}{2} \|\lambda\|_2^2
\end{align*}
and this implies by the Chernoff bound that
\begin{align*} 
P(\langle Y_n, u \rangle > (1 - \delta)x) \le \min_{\lambda \ge 0} e^{\frac{q - 1}{2} \lambda^2 - \lambda(1 - \delta)x} 
&= e^{(1 - \delta)^2x^2/2(q - 1) - (1 - \delta)^2 x^2/(q - 1)}  \\
& = e^{-(1 - \delta)^2x^2/2(q - 1)}
\end{align*}
since the optimizer is $\lambda = (1 - \delta)x/(q - 1)$. Therefore if $w = \exp((\beta/2)x^2)$ and $N(\delta)$ is a $\delta$-net of the appropriate sphere,
\[ P(W_n > w) \le |N(\delta)| \exp(-(1 - \delta)^2x^2/2(q - 1)) = |N(\delta)| w^{-(1 - \delta)^2x^2/\beta (q - 1)} \]
and as long as $\beta < 1/(q - 1)$ we can choose $\delta$ sufficiently small such that $\frac{1}{\beta(q - 1)} (1 - \delta)^2 = 1 + \epsilon$ which suffices to prove the uniform integrability estimate.
\end{proof}
\subsection{Low-temperature behavior}\label{sec:low-temp}
In this section, we prove the model must be in its low-temperature phase for $\beta > 3/(q + 1)$ in the sense it is not mutually contiguous to the product measure, and the limiting behavior of $\frac{1}{n} \log Z$ is not an analytic extension of its high temperature behavior (i.e. it is not equal to a constant); therefore the model must exhibit a phase transition at or before $3/(q + 1)$ and the high-temperature contiguity estimate from the previous section is off by at most a factor of (slightly less than) three. First we start with a basic eigenvalue calculation, which was also performed in \cite{raffaelli2005statistical}.
\begin{lemma}\label{lem:cov-eigenvalues}
Fix a number of candidates $q$ and let $\Sigma = \E X X^T$ where $X = \varphi(\pi)$ for $\pi \sim Uni(\mathfrak{S}_q)$. Then the eigenvalues of $\Sigma$ are $(q + 1)/3$ with multiplicity $q - 1$ and $1/3$ with multiplicity ${q - 1 \choose 2}$.
\end{lemma}
\begin{proof}
First, we compute the entries of $\Sigma$ --- this has also been done before in \cite{mossel2012quantitative}, but we include the details for completeness.

For convenience, we index the rows and columns of $\Sigma$ by ordered pairs $ij$ with $1 \le i < j \le q$. For entries of the form $\Sigma_{ij,k\ell}$ with all of $i,j,k,\ell$ distinct, we see that they must be zero as $\varphi(\pi)_{ij}$ and $\varphi(\pi)_{k\ell}$ are independent and each is mean zero. It remains to compute $\Sigma_{ij,k\ell}$ when there is at least one repeat among $i,j,k,\ell$. Observe that we can reduce to the case $q = 3$ because the induced ordering on three elements given by permutation of any size is the same as the induced ordering of a random permutation on those three elements, by symmetry. In the case $q = 3$, we can directly compute that
\[ \E \varphi(\pi)_{12} \varphi(\pi)_{13} = 1/3. \]
Explicitly, there are three permutations which satisfy $\pi(1) < \pi(2)$, $(1\;2\;3),(1)(2)(3),$ and $(1)(2\;3)$
and $\pi(1) < \pi(3)$ in the latter two examples, so this contributes a net value of $+1/6$ to the expectation, and the permutations with $\pi(2) < \pi(1)$ similarly contribute a net value of $+1/6$. Similarly, we can compute
\[ \E \varphi(\pi)_{12} \varphi(\pi)_{23}) = -1/3 \]
and by symmetry these values determine the remaining entries of $\Sigma$ in the case $q = 3$.

In general, by using the above solution of the case $q = 3$, we determine the entries of $\Sigma$ to be:
\begin{itemize} 
\item $+1$ on the diagonal.
\item $+1/3$ for indices of the form $(ij,ik)$ with $i < j$, $i < k$ and for indices of the form $(ij,kj)$ with $i < j$, $k < j$. 
\item $-1/3$ for indices of the form $(ij,jk)$ with $i < j < k$ and for indices of the form $(ij,ki)$ with $i < j$, $k < i$.
\item $0$ for all other entries.
\end{itemize}

Next, we can check that from the above description of $\Sigma$ that
\[ \Sigma = (1/3) I + (1/3) \sum_{r = 1}^q v_r v_r^T \]
where for each $r$ and for $i < j$, $(v_r)_{ij}$ equals $1$ if $i = r$, equals $-1$ if $j = r$, and otherwise equals zero. Observe that each vector $v_r$ is an eigenvector of $\Sigma$ with eigenvalue $1 + (q - 2)/3 = (q + 1)/3$. It follows that the eigenvalues of $\Sigma$ are $(q + 1)/3$ with multiplicity $\dim(\operatorname{span}(v_1,\ldots,v_q))$ and $1/3$ with multiplicity ${q \choose 2} - \dim(\operatorname{span}(v_1,\ldots,v_q))$. Finally, we observe that the only linear relation among the vectors $v_r$ is that $\sum_{r = 1}^q v_r = 0$: to see this is the only linear relation, observe that for any sum of the vectors $w = a_1 v_1 + \cdots + a_{q - 1} v_{q - 1}$ that the coefficient $a_i$ can be recovered from the fact that $w_{i q} = a_i$. Hence, the dimension of their span is $q - 1$ and this concludes the proof.
\end{proof}
\begin{theorem}
For any $q \ge 3$ and $\beta > 3/(q + 1)$, $\lim\inf_{n \to \infty} \frac{1}{n}(\log Z) - \log |\mathfrak{S}_q| > 0$ and the mean-field model is not mutually contiguous to the uniform measure.
\end{theorem}
\begin{proof}
By the Gibbs variational principle (Lemma~\ref{lem:gibbs-variational})
and by restricting the supremum to product measures of the form $P(X_1,\ldots,X_n) = \prod_{i = 1}^n Q(X_i)$ we see
\begin{align*} 
\log Z 
= \sup_{P} \beta \E_P[\langle Y_n, Y_n \rangle] + H(P) 
&\ge \sup_Q \left[n\beta \langle \E_Q[X], \E_Q[X] \rangle + nH(Q)\right] \\
&= n \max_s \sup_{Q : \E_Q[X] = s} \left[\beta \|s\|_2^2 + H(Q)\right]
\end{align*}
Let $\Phi(s)$ be the functional in the maximization problem above; repeating the argument from Lemma~\ref{lem:mf-eqn}, we can rewrite $\max_s \Phi(s) = \max_{\lambda} \Psi(\lambda)$ where
\[ \Psi(\lambda) = \frac{\beta}{2} \|s(\lambda)\|_2^2 - \langle \lambda, s(\lambda) \rangle + \log \sum_x e^{\langle \lambda, x \rangle} \]
where in the sum $x$ ranges over the image of $\mathfrak{S}_q$ under the embedding $\varphi$,
and $s(\lambda) = \frac{\sum_x x e^{\langle \lambda, x \rangle}}{\sum_x e^{\langle \lambda, x \rangle}}$ satisfies $\Phi(s(\lambda)) = \Psi(\lambda)$. 
As in Lemma~\ref{lem:mf-eqn} we have $\nabla_{\lambda} \Psi(\lambda) = \beta s(\lambda) s'(\lambda) - \lambda^t s'(\lambda)$ which is zero at $\lambda = 0$, and we can
compute that the Hessian at $\lambda = 0$ is given by $\beta \Sigma^2 - \Sigma$ where $\Sigma = \E_{\pi \sim \mathfrak{S}_q}[\varphi(\pi) \varphi(\pi)^T]$. This has a positive eigenvalue whenever $\beta \lambda_{max}(\Sigma) > 1$ and by Lemma~\ref{lem:cov-eigenvalues} we know that $\lambda_{max}(\Sigma) = \frac{q + 1}{3}$. 

Therefore, for $\beta > 3/(q + 1)$ the point $\lambda = 0$ is a critical point of $\Psi(\lambda)$ where the Hessian has positive eigendirections, so the maximum of $\Psi(\lambda)$ must be strictly greater than at $0$. Since $\Psi(0) = \log |\mathfrak{S}_q|$ this proves the inequality $\lim_{n \to \infty} \frac{1}{n} \log Z - \log |\mathfrak{S}_q| > 0$.
In particular, under the Gibbs measure $\E_P[\langle Y_n, Y_n \rangle] = \Omega(n)$.
Since $\langle Y_n, Y_n \rangle/n$ is bounded, this implies $|\langle Y_n, Y_n \rangle| = \Omega(n)$ with positive probability, whereas by basic concentration estimates (e.g. Markov's inequality) we know this happens with probability $o(1)$ for the product measure. This proves the sequences of measures are not mutually contiguous.
\end{proof}
Note in the statement we wrote $\lim\inf$ just because we did not prove the limit of $\frac{1}{n} \log Z$ exists.
The above result proves mutual contiguity fails;
it seems likely that in this regime, as in the $q = 3$ case, the probability of paradox is also $o(1)$. 
Proving or disproving this will probably require understanding the solutions of the mean-field equations for all values of $q$, and also perhaps finer grained (i.e. moderate deviations) behavior of the model in low temperature. 
\section{Deferred Material from Section 5}
\subsection{Proof of Theorem~\ref{thm:matching-guilbaud}}
\begin{proof}[Proof of Theorem~\ref{thm:matching-guilbaud}]
By the central limit theorem, to understand this model it will suffice to compute correlations
in a single matching. We have
\[ \E[X_1 Y_1^T] = \frac{1}{Z} \sum_{x_1,y_1} e^{\beta \langle x_1, y_1 \rangle} x_1 y_1^T \]
We group the terms by $\langle x_1, y_1 \rangle$. The possible values are $3$ (all agree), $1$ (1 disagreement), $-1$, and $3$. Therefore
\begin{align*} \E[X_1 Y_1^T] 
= \frac{1}{Z} \sum_{x_1} x_1 
&\Big(e^{3 \beta} x_1^T + e^{\beta} (2 \text{ double agreements}) \\
&\quad + e^{-\beta} (2 \text{ single agreements}) + e^{-3\beta} (-x_1^T)\Big)
\end{align*}
When $x_1 = (1,1,-1)$ the contribution to the sum is
\begin{align*}
&(1,1,-1) (e^{3\beta} (1,1,-1) + e^{\beta} ((-1,1,-1) + (1,-1,-1)) \\
&\qquad \qquad + e^{-\beta} ((1,-1,1) + (-1,1,1)) + e^{-3\beta}(-1,-1,1))^T \\
&= 2(1,1,-1) (\sinh(3\beta) (1,1,-1) + \sinh(\beta) (0,0,-2))^T \\
&= 2(1,1,-1) (\sinh(3\beta), \sinh(3\beta), -\sinh(3\beta) - 2\sinh(\beta))^T \\
&= 2\begin{bmatrix}
\sinh(3\beta) & \sinh(3\beta) & -\sinh(3\beta) - 2\sinh(\beta) \\
\sinh(3\beta) & \sinh(3\beta) & -\sinh(3\beta) - 2\sinh(\beta) \\
-\sinh(3\beta) & -\sinh(3\beta) & \sinh(3\beta) + 2\sinh(\beta)
\end{bmatrix}
\end{align*}
Note is the same for $-x_1$ and that there is a symmetry between the three coordinates. Therefore summing over all the possibilities for $x_1$ gives
\begin{align*} 
&4\begin{bmatrix}
\sinh(3\beta) & \sinh(3\beta) & -\sinh(3\beta) - 2\sinh(\beta) \\
\sinh(3\beta) & \sinh(3\beta) & -\sinh(3\beta) - 2\sinh(\beta) \\
-\sinh(3\beta) & -\sinh(3\beta) & \sinh(3\beta) + 2\sinh(\beta)
\end{bmatrix} + (2\text{ symmetrical terms}) \\
&=  4\begin{bmatrix}
3\sinh(3\beta) + 2\sinh(\beta) & -\sinh(3\beta) - 2\sinh(\beta) & -\sinh(3\beta) - 2\sinh(\beta) \\
 -\sinh(3\beta) - 2\sinh(\beta) & 3\sinh(3\beta) + 2\sinh(\beta)  & -\sinh(3\beta) - 2\sinh(\beta) \\
  -\sinh(3\beta) - 2\sinh(\beta) &  -\sinh(3\beta) - 2\sinh(\beta) & 3\sinh(3\beta) + 2\sinh(\beta)
\end{bmatrix}
\end{align*}
Since
\[ Z = 2\sum_{x_1} (\cosh(3\beta) + 2\cosh(\beta)) = 12(\cosh(3\beta) + 2\cosh(\beta)) \]
we see that 
\begin{align*}
\E[X_1 Y_1^T] 
&= \left(\frac{1}{3(\cosh(3\beta) + 2\cosh(\beta))}\right) \\
&\qquad \cdot 
\begin{bmatrix}
3\sinh(3\beta) + 2\sinh(\beta) & -\sinh(3\beta) - 2\sinh(\beta) & -\sinh(3\beta) - 2\sinh(\beta) \\
 -\sinh(3\beta) - 2\sinh(\beta) & 3\sinh(3\beta) + 2\sinh(\beta)  &  -\sinh(3\beta) - 2\sinh(\beta) \\
  -\sinh(3\beta) - 2\sinh(\beta) &  -\sinh(3\beta) - 2\sinh(\beta) & 3\sinh(3\beta) + 2\sinh(\beta)
\end{bmatrix}.
\end{align*}
By symmetry, the marginal distribution of $X_1$ is uniform over $\NAE_3$ so as before, it must be that 
\[ \E[X_1 Y_1^T] = \begin{bmatrix} 1 & -1/3 & -1/3 \\ -1/3 & 1 & -1/3 \\ -1/3 & -1/3 & 1 \end{bmatrix}. \] 
Therefore $\frac{1}{2}\E[(X_1 + Y_1)(X_1 + Y_1)^T]$ equals
\[ \begin{bmatrix} 1 + \frac{3\sinh(3\beta) + 2\sinh(\beta)}{3(\cosh(3\beta) + 2\cosh(\beta))} & -1/3 - \frac{\sinh(3\beta) + 2\sinh(\beta)}{3(\cosh(3\beta) + 2\cosh(\beta))} & -1/3 - \frac{\sinh(3\beta) + 2\sinh(\beta)}{3(\cosh(3\beta) + 2\cosh(\beta))} \\
-1/3 - \frac{\sinh(3\beta) + 2\sinh(\beta)}{3(\cosh(3\beta) + 2\cosh(\beta))} & 1 + \frac{3\sinh(3\beta) + 2\sinh(\beta)}{3(\cosh(3\beta) + 2\cosh(\beta))}&  -1/3 - \frac{\sinh(3\beta) + 2\sinh(\beta)}{3(\cosh(3\beta) + 2\cosh(\beta))}\\
 -1/3 - \frac{\sinh(3\beta) + 2\sinh(\beta)}{3(\cosh(3\beta) + 2\cosh(\beta))} &  -1/3 - \frac{\sinh(3\beta) + 2\sinh(\beta)}{3(\cosh(3\beta) + 2\cosh(\beta))} &  1 + \frac{3\sinh(3\beta) + 2\sinh(\beta)}{3(\cosh(3\beta) + 2\cosh(\beta))}
\end{bmatrix}. \]
Finally by using the Central Limit Theorem and Lemma~\ref{lem:gaussian-sgn}, it follows that the asymptotic probability of a Condorcet winner is 
\[ \frac{3}{2\pi} \arccos \left(\frac{-1/3 - \frac{\sinh(3\beta) + 2\sinh(\beta)}{3(\cosh(3\beta) + 2\cosh(\beta))}}{ 1 + \frac{3\sinh(3\beta) + 2\sinh(\beta)}{3(\cosh(3\beta) + 2\cosh(\beta))}}\right) \]
as claimed.
\end{proof}
\textbf{Acknowledgements: } We thank Mehtaab Sawhney and Jonathan Kelner for  interesting discussions about Conjecture~\ref{conj:large-q-conjecture}.
\bibliographystyle{plain}
\bibliography{bib,all}

\end{document}